\documentclass{amsart}
\usepackage{amsfonts,amsthm,amsmath}
\usepackage{amssymb}
\usepackage{amscd}
\usepackage[utf8]{inputenc}
\usepackage[a4paper]{geometry}
\usepackage{enumerate}
\usepackage[usenames,dvipsnames]{color}
\usepackage{soul}
\usepackage{array}
\usepackage{array,tabularx}
\usepackage{xypic} 
\usepackage{epic}

\usepackage[normalem]{ulem}

\usepackage{tikz}
\usetikzlibrary{arrows, intersections, calc, matrix}

\numberwithin{equation}{section}
\numberwithin{equation}{subsection}

\theoremstyle{plain}
\newtheorem{theorem}[equation]{Theorem}
\newtheorem{lemma}[equation]{Lemma}
\newtheorem{proposition}[equation]{Proposition}

\newtheorem{corollary}[equation]{Corollary}

\newtheorem{thm}[equation]{Theorem}

\newtheorem{prop}[equation]{Proposition}

\theoremstyle{definition}
\newtheorem{example}[equation]{Example}
\newtheorem{remark}[equation]{Remark}

\def\C{\mathbb C}
\def\Q{\mathbb Q}
\def\R{\mathbb R}
\def\Z{\mathbb Z}

\newcommand{\cale}{{\mathcal E}}

\newcommand{\calv}{{\mathcal V}}

\newcommand{\calt}{{\mathcal T}}

\newcommand{\cali}{{\mathcal I}}
\newcommand{\calj}{{\mathcal J}}

\newcommand{\calg}{{\mathcal G}}
\newcommand{\calF}{{\mathcal F}}\newcommand{\calf}{{\mathcal F}}
\newcommand{\calO}{{\mathcal O}}

\newcommand{\calS}{{\mathcal S}}

\newcommand{\calP}{\mathcal{P}}
\newcommand{\caln}{\mathcal{N}}

\newcommand{\bt}{{\mathbf t}}

\newcommand{\frsw}{\mathfrak{sw}}

\newcommand{\bZ}{{\mathbb{Z}}}
\newcommand{\bQ}{{\mathbb{Q}}}

\author{Tam\'as L\'aszl\'o}
\address{BCAM - Basque Center for Applied Math.,
Mazarredo, 14 E48009 Bilbao, Basque Country – Spain}
\email{tlaszlo@bcamath.org}

\author{J\'anos Nagy}
\address{Central European University, Dept. of Mathematics,  Budapest, Hungary}
\email{nagy\textunderscore janos@phd.ceu.edu}

\author{Andr\'as N\'emethi}
\address{Alfr\'ed R\'enyi Institute of Mathematics,
Hungarian Academy of Sciences,
Re\'altanoda utca 13-15, H-1053, Budapest, Hungary \newline
 \hspace*{4mm} ELTE - University of Budapest, Dept. of Geometry, Budapest, Hungary \newline \hspace*{4mm}
BCAM - Basque Center for Applied Math.,
Mazarredo, 14 E48009 Bilbao, Basque Country – Spain}
\email{nemethi.andras@renyi.mta.hu }

\title{Combinatorial duality for Poincar\'e series, polytopes and invariants of plumbed 3--manifolds}

\begin{document}



\begin{abstract}
Assume that the link of a complex normal surface singularity is a rational homology sphere.
Then its Seiberg--Witten invariant can be computed as the `periodic constant' of the topological multivariable
Poincar\'e series (zeta function). This involves  a complicated regularization procedure
(via quasipolynomials measuring the asymptotic behaviour of the coefficients).

We show that
the (a Gorenstein type) symmetry of the zeta function combined with  Ehrhart--Macdonald--Stanley reciprocity
 (of Ehrhart theory of polytopes) provide a simple expression for the periodic cosntant.
 Using these dualities we  also find a multivariable polynomial generalization of the
 Seiberg--Witten invariant, and we compute it in terms of lattice points of certain polytopes.
 
 All these invariants are also determined via lattice point counting, in this way we establish a completely general topological analogue of formulae of Khovanskii  and Morales valid for 
 singularities with non-degenerate Newton principal part. 
\end{abstract}

\maketitle

\setcounter{tocdepth}{1}
\tableofcontents

\linespread{1.2}


\pagestyle{myheadings} \markboth{{\normalsize T. L\'aszl\'o, J. Nagy, A. N\'emethi}} {{\normalsize Duality}}


\section{Introduction}

\subsection{} In order to study the topological invariants of links of complex normal surface singularities one usually uses the dual resolution  graph of the singularity. Indeed, such links
appear as plumbed 3--manifolds associated with these graphs, and the combinatorics of the
graph is an ideal source to codify the topological invariants, even if they were originally defined by
completely different methods of algebraic or differential topology. E.g., in order to determine the Seiberg--Witten invariants of a rational homology sphere link, a possible procedure is the following.
Firstly, one defines from the graph a multivariable rational function, the so--called `zeta function'
$Z(\bt)$
(it is also called the `multivariable topological Poincrar\'e series', as the topological analogue
of a `multivariable analytical Poincar\'e series' introduced by Campillo, Delgado and Gusein-Zade \cite{CDGPs,CDGEq}). This has a natural decomposition
$Z=\sum _{h\in H}Z_h$, where $H$ is the first homology of the link $M$ (and it indexes
also ${\rm Spin}^c(M)$). Then one shows that a summation associated with  special truncations
of the coefficients of $Z_h$ behaves like a quasipolynomial. And finally it turns out  that the
evaluation of the quasipolynomial at the origin (this value is called the `periodic constant
${\rm pc}(Z_h)$ of
$Z_h$') provides the Seiberg--Witten invariant
associated with the corresponding $spin^c$--structure indexed by $h$. (For details see section \ref{s:prliminaries}.)

This is a rather complex
regularization procedure. Usually it is hard to find the quasipolynomial, one needs to know all the coefficients of $Z_h$ and to understand their asymptotic behaviour.

In fact, there exists an expected reformulation/generalization
as follows. Particular examples and families suggest that there must exists an object which
is even more general and guides the above periodic constant computation as well: 
one predicts a unique canonical decomposition of $Z_h$ into a sum
$Z_h^{neg}+P^+_h$ such that $Z_h^{neg}$ is a rational function `with negative degree'
(or, with zero periodic constant) and $P^+_h$
a finite polynomial, such that $P^+_h(1)={\rm pc}(Z_h)$. In this way, one finds a
multivariable polynomial generalization of the Seiberg--Witten invariants, namely $P^+_h$.

The first goal of the present article is to provide from $Z(\bt)$  a simple expression for
${\rm pc}(Z_h)$, or for the Seiberg--Witten invariants of the link.
This is done using combination of two dualities: one of them is the topological trace of
a Gorenstein type duality, the second is the Ehrhart--Macdonald--Stanley duality of Ehrhart theory. 
It turns out that ${\rm pc}(Z_h)$ is an easy precise sum of coefficient of the
`dual' series $Z_{[Z_K]-h}$, where $[Z_K]-h \in H$ is the Gorenstein type dual of $h$. The second goal is to make the connection with
Ehrhart theory deeper: in this way we show that the decomposition $Z_h^{neg}+P^+_h$ exists indeed,
and we  determine each $P^+_h$ in language of lattice points of some polytopes. In this way
we get such an expressions for the Seiberg--Witten invariants as well.

\subsection{} Let us provide a more detailed discussion regarding the connection with lattice polytopes.

The interactions between polyhedral combinatorics and algebraic geometry is a classical and very intensive research theme. The most classical example is the toric geometry which compiles questions of algebraic geometry in terms of combinatorics of  convex cones and lattice polytopes. This leads to an infusion of combinatorial methods and formulas into the geometry and topology of complex analytic/algebraic varieties.

In the theory of complex normal surface singularities there are several  results which establish the above interaction in the case of hypersurface, or, more generally, of
isolated complete intersection singularities,  with  Newton non-degenerate principal part.
For such germs one defines Newton polyhedrons using the nontrivial monomials of the defining equations, from which several invariants are expressed, see e.g. \cite{AGZV,Kouch,MT,Var,BNnewt,Baldur,Khov,Mor,Okzeta,OkNdicis,BA}.

From a purely topological point of view, a formula by \cite{FS} provides the Casson invariant of a Brieskorn homology sphere (as the Euler characteristic of instanton Floer homology) in terms of counting lattice points inside a tetrahedron in $\R^3$ (see also \cite{NWcas}). In fact, this lattice point counting formula can also be generalized to Seifert homology spheres with at most 4 singular fibers \cite{Nic01}.
Note that for   rational homology sphere 3--manifolds  the Seiberg--Witten invariant is 
 the generalization of the Casson invariant. Furthermore, for links of Newton non-degenerate 
 hypersurface germs the Seiberg--Witten invariant can be equated with the geometric genus of the singularity proven by \cite{Baldur}, for which a lattice point counting formula in terms of the Newton polyhedron is given by \cite{MT}. However, this family  of germs is rather restrictive and does not give (even the idea) how to find  topologically the polytopes in general.
Thus, the following goal is
 very natural:

 \vspace{2mm}

 {\bf Problem:}
 {\em Find an explicit lattice point counting interpretation of the Seiberg--Witten invariant using certain `topological'  polytopes associated with the link}.

  \vspace{2mm}

   In this direction a new channel has been opened in the work of the first and the last author \cite{LN} by developing a connection between the Seiberg--Witten invariant of rational homology sphere links and equivariant multivariable Ehrhart theory of dilated polytopes:
it has been shown that equivariant multivariable Ehrhart quasipolynomials (endowed with  their Ehrhart--Macdonald--Stanley reciprocity) appear naturally in the study of multivariable
 rational `zeta functions' using the coefficient function and the periodic constant of their Taylor expansion at the origin. However, in that work the 
 strength of the reciprocity and its geometric reinterpretations were not exploited totally. 
   Though this connection suggested the existence of a {\em `polynomial -- negative degree part' decomposition of multivariable topological Poincar\'e series}, hence  a
    polynomial generalization of the Seiberg--Witten invariant, and it generated an intense activity
  (see e.g. \cite{LN,LSz,LSzdiv}), the complete explanation waited till the present manuscript.

\subsection{} The final output of the present paper is a complete answer to the Problem formulated above. The main result transforms the equivariant Ehrhart-Macdonald-Stanley reciprocity to the level of series, giving a duality between two objects:
the periodic constant of the Taylor expansion at the origin and a finite sum of coefficients of the Taylor expansion at infinity.

This, applied to the multivariable topological Poincar\'e series of normal surface singularities, gives rise to an identification of the Seiberg--Witten invariants indexed by $h$ with a finite sum of coefficients
of the `dual' series indexed by $([Z_k]-h)$. (Here $-Z_K$ is the cycle of the dualizing sheaf.)
 This can be regarded as an extension of the $Z_K$--symmetry given by the
  Riemann--Roch formula for the Euler characteristic of line bundles on the resolution. One can also be viewed as a combinatorial analog of the Laufer's duality for equivariant geometric genera of the singularity.

  Moreover, it turns out that this duality at the level of series
   provides the wished  multivariable `polynomial -- negative degree part' decomposition of the Poincar\'e series as well. It also motivates the definition of the topological polytopes associated with a plumbing graph of any rational homology sphere link, which provides
    an inclusion--exclusion lattice point counting formula for the Seiberg--Witten invariants, similar to the one of Khovanskii \cite{Khov} and Morales \cite{Mor} for the geometric genus of Newton non-degenerate isolated complete intersection singularities.

\subsection{}
The organization of the paper is the following. Section \ref{s:prliminaries} contains preliminaries about plumbing graphs, manifolds, their Seiberg--Witten invariants, and also Poincar\'e series and their counting functions and periodic constants. In section \ref{s:EMS} we shortly review the equivariant multivariable Ehrhart theory from \cite{LN,LSz}, then we formulate the Ehrhart--Macdonald--Stanley duality for rational functions. Section \ref{s:4} contains  the study of the duality for multivariable topological Poincar\'e and the  identification of the Seiberg--Witten invariants with finite sum of certain coefficients of the dual series. Section \ref{s:PP} describes   the `polynomial -- negative degree part' decomposition. It contains  a brief motivation and history of the problem, and
some examples as well. Lastly, in section \ref{s:tPolSW} we define the topological polytopes and we prove the lattice point counting formula for the polynomial part of the Poincar\'e series, in particular, for the Seiberg--Witten invariants. The results of this section are illustrated on a concrete example in Example \ref{ex:tPolSW}. We also discuss the similarities with the Khovanskii--Morales formulas using Minkowski sums of associated polyhedrons.


\subsection*{Acknowledgements}
TL was supported by ERCEA Consolidator Grant 615655 – NMST and
by the Basque Government through the BERC 2014-2017 program and by Spanish
Ministry of Economy and Competitiveness MINECO: BCAM Severo Ochoa
excellence accreditation SEV-2013-0323.

AN  was partially supported by
ERC Adv. Grant LDTBud of A. Stipsicz at R\'enyi Institute of Math., Budapest.
JN was partially supported by NKFIH Grant K119670,
JN and AN were  partially supported by  NKFIH Grant  K112735.

\section{Preliminaries}\label{s:prliminaries}

For more details regarding plumbing graphs, plumbed manifolds and their relations with normal surface singularities see \cite{BN,EN,Nfive,NJEMS,NN1,NOSZ,trieste,NPS,NWsq};
for Poincar\'e series see also \cite{CDGPs,CDGEq}.

\subsection{Plumbing graphs. Plumbed 3--manifolds.}\label{ss:PGP}
We fix a connected plumbing graph $\Gamma$
whose associated intersection matrix is negative definite. We denote the corresponding oriented
plumbed 3--manifold by
$M=M(\Gamma)$. In this article we always assume that $M$ is a  rational homology sphere,
equivalently, $\Gamma$ is a tree with all genus
decorations  zero.

We use the notation $\mathcal{V}$ for the set of vertices, $\delta_v$ for the valency of a vertex $v$,
$\mathcal{N}$ for the set of nodes (vertices with $\delta_v\geq 3$), and
$\mathcal {E}$ for the set of  end--vertices (vertices with $\delta_v=1$).

Let $\widetilde{X}$ be the plumbed 4--manifold with boundary associated with
$\Gamma$, hence $\partial \widetilde{X} = M$. Its second
homology $L:=H_2(\widetilde{X},\mathbb{Z})$ is a lattice, freely generated by the classes of 2--spheres $\{E_v\}_{v\in\mathcal{V}}$, with a negative definite intersection form  $(\,,\,)$. Furthermore, $H^2(\widetilde{X},\mathbb{Z})$
can be identified with the dual lattice  $L':={\rm Hom}_\bZ(L,\bZ)=\{l'\in L\otimes\bQ\,:\, (l',L)\in\bZ\}$.
$L'$  is generated
by the (anti)dual classes $\{E^*_v\}_{v\in\mathcal{V}}$ defined by $(E^{*}_{v},E_{w})=-\delta_{vw}$, the opposite of the Kronecker symbol.
One has the inclusions  $L\subset L'\subset L\otimes \bQ$, and  $H_1(M,\mathbb{Z})\simeq L'/L$, denoted by $H$. We write $[x]$
for the class of $x\in L'$ in $H$.

For any $h\in H$ let $r_h=\sum_vl'_vE_v\in L'$ be its unique representative
with all $l'_v\in[0,1)$.

$L'$ carries a partial ordering induced by $l'=\sum_vl'_vE_v\geq 0$ if and only if  each $l'_v\geq 0$.

The 4--manifold $\widetilde{X}$ has a complex structure.
In fact, any such $M(\Gamma)$ is the link of a complex normal surface singularity
$(X,o)$, which has a resolution $\widetilde{X}\to X$ with
resolution graph $\Gamma$ (see e.g. \cite{Nfive}),
and the complex analytic smooth surface $\widetilde{X}$ as a $C^\infty$--manifold is the plumbed 4--manifold
 associated with $\Gamma$.
Let $K\in L'$ be its canonical cycle.    Though the complex structure is not unique,
 $K$ is determined topologically by $L$ via the adjunction formulae $(K+E_v,E_v)+2=0$ for all $v$.
In some cases, it is more convenient to use $Z_K:=-K\in L'$, hence, by adjunction formulae,
 \begin{equation}\label{ZK}
 Z_K-E=\sum_{v\in\mathcal{V}}(\delta_v-2)E^*_v.
 \end{equation}
In the sequel   $|A|$ denotes the cardinality of the set $A$.

\subsection{The series $Z(\mathbf{t})$.}
The  \emph{multivariable topological Poincar\'e series} is the
Taylor expansion $Z(\mathbf{t})=\sum_{l'} z(l')\bt^{l'}
\in\bZ[[L']] $ at the  origin of the `zeta-function'
\begin{equation}\label{eq:1.1}
f(\mathbf{t})=\prod_{v\in \mathcal{V}} (1-\mathbf{t}^{E^*_v})^{\delta_v-2},
\end{equation}
where
$\bt^{l'}:=\prod_{v\in \mathcal{V}}t_v^{l'_v}$  for any $l'=\sum _{v\in \mathcal{V}}l'_vE_v\in L'$ ($l'_v\in\bQ$), cf. \cite{CDGPs,CDGEq,NPS}.
It decomposes as $Z(\mathbf{t})=\sum_{h\in H}Z_h(\mathbf{t})$, where $Z_h(\mathbf{t})=\sum_{[l']=h}z
 (l')\bt^{l'}$. The expression
(\ref{eq:1.1}) shows that  $Z(\mathbf{t})$ is supported in the \emph{Lipman cone} $\mathcal{S}':=\mathbb{Z}_{\geq0}\langle E^{*}_{v}\rangle_{v\in\mathcal{V}}$.
Since  $(\,,\,)$ is negative definite, all the entries of
$E_v^*$ are strict positive, hence $\calS'\subset
\{\sum_vl'_vE_v\,:\, l'_v>0\}\cup\{0\}$.
Thus (cf. \cite[(2.1.2)]{NJEMS}) for any $x$,
\begin{equation}\label{eq:finite}
\{l'\in \calS'\,:\, l'\not\geq x\} \ \ \mbox{is finite}.\end{equation}
Fix $h\in H$.  We define a `counting function'
of the coefficients of $Z_h$ by
\begin{equation}\label{eq:countintro}
Q_h: L'_h:=\{x \in L'\,:\, [x]=h\}\to \bZ, \ \ \ \
Q_{h}(x)=\sum_{l'\ngeq x,\, [l']=h} z(l').
\end{equation}

The appearance  of this type of  truncation $\{l'\ngeq x,\, [l']=h\}$ for the
counting function
in the above sum is motivated by the results  below, see
e.g. \ref{eq:SUM} (or \cite{NCL}).

\subsection{Seiberg--Witten invariants of $M$.}
 Let
$\widetilde{\sigma}_{can}$ be the {\it canonical
$spin^c$--structure on $\widetilde{X}$} identified by $c_1(\widetilde{\sigma}_{can})=Z_K$,
and let $\sigma_{can}\in \mathrm{Spin}^c(M)$
 be its restriction to $M$, called the {\it canonical
$spin^c$--structure on $M$}. $\mathrm{Spin}^c(M)$ is an $H$--torsor
 with action denoted by $*$.

 We denote by $\mathfrak{sw}_{\sigma}(M)\in \bQ$ the
\emph{Seiberg--Witten invariant} of $M$ indexed by the $spin^c$--structures $\sigma\in {\rm Spin}^c(M)$ (cf. \cite{Lim, Nic04}).
(We will use the sign convention of \cite{BN,NJEMS}.)

In the last years,  several combinatorial expressions were established for the Seiberg--Witten invariants.   For rational homology spheres,
Nicolaescu \cite{Nic04} showed  that $\mathfrak{sw}(M)$ is
equal to the Reidemeister--Turaev torsion normalized by the Casson--Walker invariant. In the case when $M$ is a negative
definite plumbed rational homology sphere, a combinatorial formula for Casson--Walker invariant in terms of the plumbing graph can be found in Lescop
\cite{Lescop}, and  the Reidemeister--Turaev torsion is determined by N\'emethi and Nicolaescu \cite{NN1} using Dedekind--Fourier sums.

A different  combinatorial formula of $\{\mathfrak{sw}_\sigma(M)\}_\sigma$ was proved  in \cite{NJEMS} using qualitative properties of the coefficients of the series $Z(\mathbf{t})$.

\begin{thm} \label{th:JEMS}\ \cite{NJEMS}
For any $l'\in -K+ \textnormal{int}(\mathcal{S}')$ (where $\textnormal{int}(\mathcal{S}')=
\Z_{>0}\langle E^*_v\rangle_{v\in \calv}$)
\begin{equation}\label{eq:SUM} - Q_{[l']}(l')=
\frac{(K+2l')^2+|\mathcal{V}|}{8}+\mathfrak{sw}_{[-l']*\sigma_{can}}(M).
\end{equation}
\end{thm}
\noindent If we fix $h\in H$ and we write  $l'=l+r_{h}$ with $l\in L$,
then the right--hand side of (\ref{eq:SUM}) is
 a multivariable quadratic polynomial on $L$,
 a fact which
  will be exploited conceptually  next.

\subsection{\bf Periodic constants.}\label{ss:perConst}
 One of the key tools of the present article is an invariant associated with series motivated by properties of Hilbert--Samuel functions used in algebraic geometry and
singularity theory.
 It creates a bridge with Ehrhart theory and the properties of its quasipolynomials.
It is called the {\it periodic constant} of the series.
For  one--variable series they were introduced
in \cite{NOk,Ok}, see also \cite{BN},
the multivariable generalization is treated in  \cite{LN}.

Let $S(t)=\sum_{l\geq0}c_l t^l \in
\mathbb{Z}[[t]]$ be a formal power series with one variable.
Assume that for some $p\in \mathbb{Z}_{>0}$ the counting function
$Q^{(p)}(n):=\sum_{l=0}^{pn-1}c_l$ is a polynomial $\mathfrak{Q}^{(p)}$
in $n$. Then the constant term $\mathfrak{Q}^{(p)}(0)$ is independent of $p$ and it is called the \emph{periodic
constant} $\mathrm{pc}(S)$ of the series $S$. E.g.,
if $S(t)$ is a finite polynomial, then $\mathrm{pc}(S)$ exists and it equals
$S(1)$.
If the coefficients of $S(t)$ are given by a Hilbert function $l\mapsto c(l)$,
which admits a Hilbert polynomial $H(l)$ with $c(l)=H(l)$ for $l\gg 0$, then one shows that
$S^{reg}(t)=\sum_{l\geq 0}H(l)t^l$ has zero periodic constant and
$\mathrm{pc}(S)=\mathrm{pc}(S-S^{reg})+\mathrm{pc}(S^{reg})=(S-S^{reg})(1)$,
measuring the difference between the Hilbert function and Hilbert polynomial.

For $\mathrm{pc}(S)$, valid for the Taylor expansion of special rational functions,
see  \ref{ss:PPmothist}.

For the multivariable case we consider a (negative) definite lattice
$L=\Z\langle E_v\rangle_v$,
 its dual lattice $L'$, a series $S(\mathbf{t})\in \bZ[[L']]$
(e.g. $Z(\bt)$), and its well-defined counting function $Q_h=Q_h(S(\bt))$ as in (\ref{eq:countintro}) for fixed $h\in L'/L$. (In fact, the definition extends to the case of any
two free $\Z$--modules $L\subset L'$ of the same rank with $H=L'/L$, the setup of Section
\ref{s:EMS}.)
Assume that there exist
 a real cone $\mathcal{K}\subset L'\otimes\mathbb{R}$ whose affine closure is  top--dimensional,  $l'_* \in \mathcal{K}$,
 a sublattice $\widetilde{L} \subset L$ of finite index,
and  a quasipolynomial $\mathfrak{Q}_h(l)$ ($l\in \widetilde{L}$)
(written also as  $\mathfrak{Q}^{{\mathcal K}}_h(l)$)
such that $ Q_h(l+r_h)=\mathfrak{Q}_h(l)$  for any
$l+r_h\in (l'_* +\mathcal{K})\cap (\widetilde {L}+r_h)$. Then we say that
 the counting function $Q_h$ (or just $S_h(\mathbf{t})$)
 admits a quasipolynomial in $\mathcal{K}$, namely $\mathfrak{Q}_h(l)$,  and
also  an (equivariant, multivariable)  {\em periodic constant}
associated with $\mathcal{K}$,  which is defined by
\begin{equation}\label{eq:PCDEF}
\mathrm{pc}^{\mathcal{K}}(S_h(\bt)) :=\mathfrak{Q}_h(0).
\end{equation}
 The definition does not depend on the choice of the sublattice $\widetilde L$, 
 which corresponds to the choice of
 $p$ in the one--variable case.
 This is responsible for the name `periodic' in the definition.
 The definition is independent of the choice of $l'_*$ as well.

By the general theory of multivariable Ehrhart-type quasipolynomials
(counting lattice points with special coefficients
 in polytopes attached to $Z(\mathbf{t})$)
 one can construct a conical chamber decomposition of the space $L'\otimes\mathbb{R}$, such that each cone satisfies the above definition
 (hence provides a periodic constant), for details see \cite{LN} or \cite{SzV}.
This decomposition, in principle,  divides $\mathcal{S}'_{\mathbb{R}}
:= {\mathbb R}_{\geq 0}\langle E^*_v\rangle _{v\in \calv}$
into several sub--cones (hence, providing different quasipolynomials and periodic constants associated with these sub--cones  of $\mathcal{S}'_{\mathbb{R}}$).
However, Theorem \ref{th:JEMS} guarantees that this is not the case,
the whole $\mathcal{S}'_{\mathbb{R}}$ admits a unique quasipolynomial, in particular,
 a periodic constant (cf. also with \cite{LSz}).
Hence, Theorem \ref{th:JEMS} reads as follows.
\begin{thm}\label{th:NJEMSThm} \ \cite{NJEMS}
The counting function
 of $Z_h(\bt)$ in the cone $S'_{\mathbb{R}}$
admits  the (quasi)polynomial
\begin{equation}\label{eq:SUMQP} \mathfrak{Q}_{h}(l)=
-\frac{(K+2r_h+2l)^2+|\mathcal{V}|}{8}-\mathfrak{sw}_{-h*\sigma_{can}}(M),
\end{equation}
whose periodic constant is
\begin{equation}\label{eq:SUMQP2}
\mathrm{pc}^{S'_{\mathbb{R}}}(Z_h(\mathbf{t}))=\mathfrak{Q}_h(0)=
-\mathfrak{sw}_{-h*\sigma_{can}}(M)-\frac{(K+2r_h)^2+|\mathcal{V}|}{8}.
\end{equation}
\end{thm}
The right--hand side of (\ref{eq:SUMQP2}) with opposite sign
is called the  {\it $r_h$--normalized Seiberg--Witten invariant} of $M$, and it is denoted by
$\mathfrak{sw}_h^{norm}(M)$.

\section{Ehrhart--Macdonald--Stanley duality for rational functions}\label{s:EMS}

\subsection{}
We wish to apply equivariant Ehrhart theory for the computation of the periodic
constant  of the topological Poincar\'e series.
For this we review some results from the Ehrhart theory. We will use a similar setup
as in section \ref{s:prliminaries} to make the comparison easier:
we will fix a free $\Z$--module $L$, another one  $L'\supset L$ of the same rank
(but in this Ehrhart context $L'$ is not
necessarily the dual of $L$, in fact $L$ carries no intersection form at all).
We write $H$ for $L'/L$, $d$ for the order of $H$,
and we fix a bases  $\{E_v\}_{v\in\calv}$ in $L$.

We consider multivariable rational functions (in variables $\bt^{L'}$) of type
\begin{equation}\label{eq:func}
z(\bt)=\frac{\sum_{k=1}^r\iota_k\bt^{b_k}}{\prod_{i=1}^n
(1-\bt^{a_i})},
\end{equation}
where  $\{\iota_k\}_{k=1}^r \in\mathbb{Z}$, $\{b_k\}_{k=1}^r, \ \{a_i\}_{i=1}^n\in L'$ and for any
$l'=\sum_vl'_vE_v\in L'$ we write $\bt^{l'}=t_1^{l'_1}\dots
t_{s}^{l'_{s}}$.
(Though $l'_v\in{\mathbb Q}$  we still call the elements of $\Z[L']$ `polynomials',
and fractions of `polynomials' `rational functions'.)
  We also assume that
all the coordinates $a_{i,v}$ of $a_i$  are strict positive.

We consider the Taylor expansion $Tz({\bt})$ of $z(\bt)$ at the origin
$$Tz(\bt)=\sum_{l'}z(l')\bt^{l'}\in \Z[[\bt^{1/d}]][\bt^{-1/d}]:=\Z[[t_1^{1/d},\ldots, t_s^{1/d}]]
[t_1^{-1/d},\ldots, t_s^{-1/d}],$$
and also define the Taylor expansion of $z(\bt)$ at infinity
$$T^{\infty}z(\bt)=\sum_{\tilde{l}}r(\tilde{l})\bt^{\tilde{l}}\in \Z[[\bt^{-1/d}]][\bt^{1/d}].$$
$T^\infty z$  is obtained by
the substitution  $\mathbf{s}=1/\bt$ into the Taylor expansion at $\mathbf{s}=0$ of the function $z(1/\mathbf{s})$.
E.g., if $z(t)=t^2/(1-t)$, then 
$T^\infty z(t)=-t(1+t^{-1}+t^{-2}+\cdots)$.

\subsection{} The series  $z$, $Tz$ and $T^{\infty}z$
 have equivariant decompositions $\sum_{h}z_h$, $\sum_h Tz_h$ and $\sum_h (T^{\infty}z)_h$ ($h\in H$)
 with respect to  $H=L'/L$. They
  are defined similarly as  the decomposition  of $Z$ in \ref{ss:PGP}.
   (Note that the operation $Tz\mapsto T^\infty z$, defined via the original $z$,
  preserves all the $h$--components.)

We also might eliminate some of the variables:
 for any subset
 $\cali\subset \calv$ we substitute $t_i=1$  in $z(\bt)$  for all $i\notin \mathcal{I}$;
 in this way we obtain $z(\bt_{\cali})$. We call this procedure `reduction'.
 This procedure at the level of $Tz$ is a summation of some of the coefficients.
 Since the series associated with the denominator of $z$
 is supported in the cone $\Z_{\geq 0}\langle a_i\rangle$, the summation is finite.

 Note that the $H$--decomposition of the restricted functions are not well--defined. That is,
 from the restricted function $z(\bt_{\mathcal{I}}):=z(\bt)|_{t_i=1,i\notin \mathcal{I}}$
 in general one cannot recover anymore the restriction of the $H$--component $z_h(\bt)$.
 (That is, from the exponent of $\bt_\cali^{l'}$ one cannot recover $[l']\in H$.
Here, and in the sequel,  for any $l'\in L'$ we write $\bt^{l'}_\cali$ for $\prod_{v\in \cali}t_v^{l'_v}$.)
Hence, in the presence of decomposition and reduction the only well--defined object is
$(z_h)|_{t_i=1,i\notin \mathcal{I}}$,  the reduction of $z_h$, which will be denoted by $z_h(\bt_\cali)$.
Furthermore, if $\pi_\cali$ is the natural projection associated with the reduction (elimination of the
$\cali$--components), then $\pi_{\mathcal{I}}(L')/\pi_{\mathcal{I}}(L)$ usually is not isomorphic to $H$.
Hence, we  keep $H$ as an index set and we never consider $\pi_{\mathcal{I}}(L')/\pi_{\mathcal{I}}(L)$.

\subsection{} We fix  $h\in H$ and $\cali\subset \calv$.
We define two functions associated with the coefficients of $Tz_h(\bt_{\mathcal{I}})$: the first is called the
{\em counting function} (cf. (\ref{eq:countintro}))
\begin{equation}\label{eq:countintro2}
Q_{h,\cali}: L'_{h}:=\{x\in L'\,:\, [x]=h\}\to \bZ, \ \ \ \ Q_{h,\cali}(x):=\sum_{l'|_\cali\ngeq x |_\cali,
\, [l']=[x]} z(l').
\end{equation}
The definition selects only the coefficients of $Tz_h(\bt_\cali)$, hence if we write
$Tz_h(\bt_\cali)$ as $\sum_{k'\in \pi_\cali L'} z^{(h)} (k')\bt_{\cali}^{k'}$, then
$Q_{h,\cali}(x)=\sum_{k'\not\geq x|_{\cali}}z^{(h)}(k')$.
It depends only on $x|_{\cali}$.
This truncation and counting function does not (naturally) appear in Ehrhart theory, but
this  is what is imposed by Theorem \ref{th:JEMS}.

Our second function is called the      {\em modified counting function}; it is  defined by
\begin{equation}\label{eq:modcountintro}
q_{h,\cali}: L'_{h}\to \bZ, \ \ \ \ q_{h,\cali} \, (x):=\sum_{l'|_\cali \prec\,   x|_\cali, \ [l']=[x]} \ z(l'),
\end{equation}
where for any $a,b\in \mathbb{R}^{|\cali|}$ we say that $a\prec b$ if $a_v<b_v$
for {\it all} $v\in \cali$.
Similarly as above, one also has $q_{h,\cali} \, (x)=\sum_{k'\prec  x|_{\cali}}z^{(h)}(k')$.
For any $h$ and $\cali$ by inclusion--exclusion principle
\begin{equation}\label{eq:InEx}
Q_{h,\cali}(x)=\sum_{\emptyset \not=\calj\subset \cali} \ (-1)^{|\calj|+1} q_{h,\calj}(x).\end{equation}
This modified counting function will be related to the usual counting functions from the Ehrhart theory. It satisfies several useful properties (e.g. convexity, reciprocity, see below).

\subsection{Ehrhart theory of polytopes associated with the denominator of $z$}\label{sec:denom}
In this subsection we follow the equivariant version from \cite{LN,LSz}.
Associated with
the vectors $\{a_i\}_{i=1}^n$ (the exponents in the denominator of $z(\bt)$)
we define two objects. Firstly,
let $\mathfrak{l}:\R^n\to L'\otimes\R$ be the map given by $\mathfrak{l}(\mathbf{x})=\sum_{i=1}^n x_ia_i$ and consider the representation $\rho:\Z^n\to H$ defined by the composition
$\Z^n\stackrel{\mathfrak{l}|_{\Z^{n}}}{\longrightarrow} L'\to L'/L$.

Then, the vectors $\{a_i\}_{i=1}^n$ and any
  $\cali\subset \calv, \cali\neq \emptyset$ (which might vary, cf. (\ref{eq:InEx}))
determines  the family of
closed  dilated convex polytopes indexed (or, dilated by) $l'=\sum_v l_v' E_v\in L'$
\begin{equation}\label{def:polytope}
\calP^{(l')}_\cali :=\{{\bf x}\in (\R_{\geq 0})^n\,:\,
\sum_ix_ia_{i,v} \leq  l_v' \ \ \mbox{for all $v\in \cali$} \}.
\end{equation}
$ \calP^{(l')}_\cali $  depends only on $l'_\cali:=l'|_{\cali}$.
We denote by $\calF^{(l')}_\cali$ the set of facets of $\calP^{(l')}_\cali$.

We assume that
$l'_\cali$ moves in some region of $\pi_\cali(L')$
in such a way that $\calP^{(l')}_\cali$ stays combinatorially stable. In such a case
we can associate to the stable combinatorial type of the dilated polytopes the set (dilated family)
 of facets $\calf^{(l')}_\cali$. Moreover, any choice of a subset of facets in a fixed  stable
 topological type provides  a `stable' (dilated)  subset of facets $ \calg^{(l')}_\cali$
in each $ \calf^{(l')}_\cali$; we denote this choice by  $\calg_\cali\subset \calf_\cali$.
Furthermore,
for any $h\in H$ and choice  $\calg_\cali\subset \calf_\cali$ we consider the
counting function of specially chosen lattice points identified  by
\begin{equation}\label{eq:QPol}
\mathcal{Q}_{h,\calg_\cali}(l')
:=\mbox{cardinality of}\ \big(\big(\calP^{(l')}_\cali \setminus \calg^{(l')}_\cali
\big)\cap \rho^{-1}(h)\big).\end{equation}
According to the  equivariant Ehrhart theory,
  applied to the dilated polytopes $\calP^{(l')}_\cali$,
and for any $h\in H$ and $\calg_\cali$,
the counting function   $\mathcal{Q}_{h,\calg_\cali}(l'_\cali)$
is a quasipolynomial,  whenever $l'_\cali$ moves in some region of $\pi_\cali(L')$
in such a way that $\calP^{(l')}_\cali$ stays combinatorially stable.

The inequalities of (\ref{def:polytope}) can be viewed as a vector partition
$\sum_{i=1}^n x_i\cdot a_i|_\cali+
\sum_{v\in \cali}y_v\cdot  E_v|_\cali=l'_\cali$ with $x_i, y_v\geq 0$. Hence, by \cite{SzV,LSz},
 the variance of the combinatorial
type of $\calP^{(l')}_\cali$ is determined by the following chamber decomposition of
$\pi_\cali(L'\otimes {\mathbb R})
={\mathbb R}^{|\cali|}$: let $\mathcal{B}_\cali$ be the set of all bases
$\sigma\subset \{a_i|_\cali, E_v|_\cali  :  i\in\{1,\dots,n\}, v\in\cali\}$ of ${\mathbb R}^{|\cali|}$.
Then a (big, open) chamber $\mathfrak{c}$ is a connected component of ${\mathbb R}^{|\cali|}\setminus \cup_{\sigma\in\mathcal{B}_\cali}\partial\R_{\geq 0}\sigma$, where $\partial\R_{\geq 0}\sigma$ is the boundary of the closed cone $\R_{\geq 0}\sigma$.

This construction shows  that the chamber decomposition is independent of the choices of
$h\in H$ and $\calg_\cali\subset \calf_\cali$.

Therefore, if $l'_\cali$ moves in such a chamber $\mathfrak{c}$, then the above counting function of
lattice points $\mathcal{Q}_{h,\calg_\cali}(l'_\cali)$ is a quasipolynomial, denoted by
$\mathcal{Q}_{h,\calg_\cali}^{\mathfrak{c}}$. One can  extend $\mathcal{Q}_{h,\calg_\cali}^{\mathfrak{c}}$
continuously to the closure  $\overline{\mathfrak{c}}$ of $\mathfrak{c}$ as a quasipolynomial;
all these
extensions glue together to a continuous piecewise quasipolynomial of ${\mathbb R}^{|\cali|}$.
This  piecewise  quasipolynomial is the expression ${\mathcal Q}_{h,\calg_\cali}$ from (\ref{eq:QPol}).
(This means that  ${\mathcal Q}_{h,\calg_\cali}(l'_\cali)$ in terms of
$\{\mathcal{Q}_{h,\calg_\cali}^{\mathfrak{c}}\}_{\mathfrak{c}}$ can be  redefined as follows:
for any  $l'_\cali \in {\mathbb R}^{|\cali|}$
  first find a chamber $\mathfrak{c}$ such that
$l'_\cali\in \overline{\mathfrak{c}}$ and then set
 ${\mathcal Q}_{h,\calg_\cali}(l'_\cali):=\mathcal{Q}_{h,\calg_\cali}^{\mathfrak{c}}(l'_\cali)$.)

By the   {\em equivariant Ehrhart--Macdonald--Stanley reciprocity law}, for any fixed $h$, $\calg_\cali\subset \calf_\cali$
and chamber $\mathfrak{c}$
with $l'_\cali\in \mathfrak{c}$ one has
\begin{equation}\label{eq:EMS}
\mathcal{Q}_{h,\calg_\cali}^{\mathfrak{c}}(l'_\cali) =(-1)^n\cdot
\mathcal{Q}_{-h,\calf_\cali\setminus \calg_\cali}^{\mathfrak{c}}(-l'_{\cali}).
\end{equation}
(We warn the reader that usually the parameter $-l'_\cali$ is included in a different chamber
than $\mathfrak{c}$, that is
$\mathcal{Q}_{-h,\calf_\cali\setminus \calg_\cali}^{\mathfrak{c}}(-l'_{\cali})\not=
{\mathcal Q}_{-h,\calf_\cali\setminus \calg_\cali}(-l'_{\cali})$;
in (\ref{eq:EMS}) in both sides we use the same chamber $\mathfrak{c}$, and from the pair
$l'_\cali$, $-l'_\cali$ only one of them can be in $\overline{\mathfrak{c}}$
provided that $l'_\cali\not=0$.)

\subsection{}
A distinguished subset of facets  $\calg_\cali\subset \calf_\cali$ is defined as  the
 coordinate facets $\{\calP^{(l')}_\cali\cap \{x_i=0\}\}_{i=1}^n$. We denote it by  $\calg_\cali^{co}$.

 Then the theory above has the following consequences regarding
 the modified counting function $q_{h,\cali}$ of $Tz_h (\bt_\cali)$ defined in (\ref{eq:modcountintro}):
\begin{corollary}\label{cor:NEW} We fix $h$ and $\cali$.

(a)  $q_{h,\cali}$ is a piecewise quasipolynomial, which can be written as
\begin{equation}
q_{h,\cali}(l')=\sum_k\iota_k\cdot\mathcal{Q}_{h-[b_k], \calf_\cali\setminus \calg^{co}_{\cali}}
(l'_\cali-b_k|_\cali).
\end{equation}
\ \ (b) For a fixed chamber $\mathfrak{c}$ of ${\mathbb R}^{|\cali|}$ define the quasipolynomial
\begin{equation}\label{eq:qbar}
q^{\mathfrak{c}}_{h,\cali}(l'):=
\sum_k\iota_k\cdot\mathcal{Q}^{\mathfrak{c}}_{h-[b_k],
 \calf_\cali\setminus \calg^{co}_{\cali}}(l'_\cali-b_k|_\cali).
\end{equation}
Then $q_{h,\cali}$ is a quasipolynomial on (the closure of)
  $\cap_k(b_k|_\cali+\mathfrak{c})$, namely, for  $l'_\cali\in
  \cap_k(b_k|_\cali+\mathfrak{c})$,
$q_{h,\cali}(l')=q^{\mathfrak{c}}_{h,\cali}(l')$.

(c) For any fixed chamber $\mathfrak{c}$
the modified counting function $q_{h,\cali}$ admits a quasipolynomial
in the sense of Subsection \ref{ss:perConst} ${\mathfrak{q}}^{\mathfrak{c}}_{h,\cali}$,
which satisfies for $l'=l+r_h$ the identity
${\mathfrak{q}}^{\mathfrak{c}}_{h,\cali}(l)=q^{\mathfrak{c}}_{h,\cali}(l')$,
  and a
periodic constant  $\mathrm{pc}^{\mathfrak{c}}(  q_{h,\cali})=
{\mathfrak{q}}^{\mathfrak{c}}_{h,\cali}(0)=q^{\mathfrak{c}}_{h,\cali}(r_h)$
associated with  the chamber $\mathfrak{c}$.
\end{corollary}
This $\mathrm{pc}^{\mathfrak{c}}(  q_{h,\cali})$ will be denoted by
$\mathrm{mpc}^{\mathfrak{c}}(Tz_h (\bt_\cali))$.
We call it the
{\em modified periodic constant} of $Tz_h (\bt_\cali)$ associated with $h$, $\cali$
and the chamber $\mathfrak{c}$.
(The terminology and the notation emphasize the presence of  different cuts in the counting functions.)

\subsection{}
In general, the computation of quasipolynomials and their  periodic constants  (either modified or not)
is hard: it measures the asymptotic
behaviour of the coefficients in a certain cone.
The next result based on the
equivariant Ehrhart--Macdonald--Stanley reciprocity (\ref{eq:EMS}) shows that (under some condition)
the modified periodic constant equals with a
{\em finite sum given by the coefficients of the Taylor expansion at infinity}.
\begin{thm}\label{thm:mpc} Fix $h$ and $\cali$ as above. Assume that there exists a chamber $\mathfrak{c}$
such that
$b_k|_\cali\in\mathfrak{c}$ for all $k$.
 Write the $h$--component of the Taylor expansion at infinity as
$(T^{\infty}z)_h(\bt_\cali)=
\sum_{\tilde{l}}r^{(h)}_{\cali}(\tilde{l})\bt_\cali^{\tilde{l}}$.
Then
\begin{equation}\label{eq:mpc}
\mathrm{mpc}^{\mathfrak{c}}(Tz_h (\bt_\cali))=\sum_{\tilde{l}\geq 0}r^{(h)}_{\cali}(\tilde{l}).
\end{equation}
\end{thm}
\begin{proof} By Corollary \ref{cor:NEW} the function $q_{h,\cali}$
 on the set $\cap_k (b_k|_\cali+\mathfrak{c})$
is
$\sum_k\iota_k\cdot\mathcal{Q}^{\mathfrak{c}}_{h-[b_k],
 \calf_\cali\setminus \calg^{co}_{\cali}}(l'_\cali-b_k|_\cali)$.
Hence, by definitions,
$\mathrm{mpc}^{\mathfrak{c}}(Tz_h (\bt_\cali))=
{\mathrm pc}^{\mathfrak{c}}(q_{h,\cali})$ exists and equals
$$\sum_k\iota_k\cdot\mathcal{Q}^{\mathfrak{c}}_{h-[b_k],
 \calf_\cali\setminus \calg^{co}_{\cali}}((r_h-b_k)|_\cali).$$
 First we claim that
$$ \mathcal{Q}^{\mathfrak{c}}_{h-[b_k],
 \calf_\cali\setminus \calg^{co}_{\cali}}((r_h-b_k)|_\cali)=
 \mathcal{Q}^{\mathfrak{c}}_{h-[b_k],
 \calf_\cali\setminus \calg^{co}_{\cali}}(-b_k|_\cali).$$
 Indeed, if $[\sum _ix_ia_i]=h-[b_k]$ and
 $(\sum_i x_ia_i)_v< (r_h-b_k)_v$  for all $v\in\cali$ (where the strictness of the
 inequality is imposed by the boundary condition  $\calf_\cali\setminus \calg^{co}_{\cali}$)
 then necessarily $(\sum_i x_ia_i)_v< (-b_k)_v$ also holds for $v\in\cali$. That is, if
 $[a]=h$ and $a_v<(r_h)_v$ then $a_v<0$. To see this write $a=r_h+l$ with $l\in L$,
 then $l_v<0$ hence $l_v\leq -1$, and
 $a_v\leq (r_h)_v-1<0$.

 On the other hand,  since $b_k\in \mathfrak{c}$, from (\ref{eq:EMS}) one has
$$ (-1)^n\cdot \mathcal{Q}^{\mathfrak{c}}_{-h+[b_k],\calg^{co}_{\cali}}(b_k|_\cali)=
 \mathcal{Q}^{\mathfrak{c}}_{h-[b_k],
 \calf_\cali\setminus \calg^{co}_{\cali}}(-b_k|_\cali).$$
 The expression $\mathcal{Q}^{\mathfrak{c}}_{-h+[b_k],\calg^{co}_{\cali}}(b_k|_\cali)$
 counts solutions  of $\sum_ix_ia_{i,v}\leq b_{k,v}$ for all $v\in\cali$ under the restrictions
 $[b_k-\sum_ix_ia_i]=h$ and  $x_i>0$ for all $i$. On the other hand,
 in the expansion at infinity,
 $$(-1)^n\cdot \frac{\bt^{b_k}}{\prod_i(1-\bt^{a_i})}\stackrel{T^\infty}{\longrightarrow}
 \bt^{b_k} \cdot \sum_{x_i>0}\bt^{- \sum_ix_ia_i}.$$
 Hence
 $$\sum_k\iota_k\cdot(-1)^n\cdot \mathcal{Q}^{\mathfrak{c}}_{-h+[b_k],\calg^{co}_{\cali}}(b_k|_\cali)
 =\sum_{\tilde{l}\geq 0}r^{(h)}_{\cali}(\tilde{l}).$$
\end{proof}

\section{Duality for multivariable  Poincar\'e series}\label{s:4}

\subsection{} In this section, we will return  to the situation of section \ref{s:prliminaries}:
we consider the rational function $f(\bt)$ from (\ref{eq:1.1})
associated with a connected negative definite plumbing graph $\Gamma$, its expansion at the origin
$Z(\bt)$, its expansion at infinity $T^\infty f(\bt)$,  and their equivariant versions and reductions.
We also use  all the notations from section \ref{s:EMS} regarding the counting and modified counting functions
associated with $f$ and $Z$, $\mathcal{S}_{\mathbb R}'$
is the real Lipman cone, and for any $\cali$ we consider the projection
  $\pi_{\cali}:\mathbb{R}\langle E_v\rangle_{v\in\mathcal{V}}
  \to \mathbb{R}\langle E_v\rangle_{v\in\cali}={\mathbb R}^{|\cali|}$, denoted also as
 $x\mapsto x|_{\cali}$.
The  {\it projected (real) Lipman cone} is the cone in ${\mathbb R}^{|\cali|}$  defined as
$\pi_{\cali}(S'_{\mathbb{R}})$.

We also review  a reduction  analogue of Theorem \ref{th:NJEMSThm},
which is the prototype of several results showing that for certain
choices of $\cali$, the reduction $Z(\bt_\cali)$ still preserves the  `Seiberg--Witten information' (for more see \cite{LNN}).
\begin{thm}\label{th:RedThm} \ \cite{LN}
The counting function
 of $Z_h(\bt_{\caln})$ in the cone $\pi_{\caln}(S'_{\mathbb{R}})$
admits  a quasipolynomial and a periodic constant,   and
$$\mathrm{pc}^{\pi_{\mathcal{N}}(S'_{\mathbb{R}})}(Z_h(\mathbf{t}_{\mathcal{N}}))=\mathrm{pc}^{S'_{\mathbb{R}}}(Z_h(\mathbf{t}))=
-\mathfrak{sw}_{-h*\sigma_{can}}(M)-\frac{(K+2r_h)^2+|\mathcal{V}|}{8}.$$
\end{thm}
Such a  result has the following  advantages:
the number of reduced variables (i.e., in this case, the number of nodes) usually is
considerably less than the number of vertices, a fact which reduces the complexity of the calculations.
Moreover, the reduced series reflects more conceptually
the complexity of the manifold $M$ (using  only one variable for each Seifert 3-manifold piece in its JSJ--decomposition).  Furthermore,
the reduced series can be compared/linked with other (geometrically
or analytically defined)
 objects as well (see e.g. \cite{BN,NPS}).

The above theorem has several generalizations, however
the number of variables cannot be decreased arbitrarily,
it is obstructed by the  normalized Seiberg--Witten invariants of
the complementary subgraph $\Gamma(\calv\setminus \cali)$, for details see
\cite{LNN}.

\subsection{} As we already mentioned, the {\bf modified counting function} has some
additional nice properties (compared with the original counting function).
Regarding it,  in the sequel we will use several results
 from \cite{LSz} (where  $q_{h,\cali}$ associated with $f$ is  called the `coefficient function', since
 $q_{h,\cali} \, (l'_0)$ is the coefficient of $\mathbf{t}^{l'_0}_{\cali}$ in the Taylor expansion of $f_h(\mathbf{t}_{\cali})\cdot\prod_{v\in\cali}
 \mathbf{t}^{E_v}_{\cali}/(1-\mathbf{t}^{E_v}_{\cali})$).

\subsection{`Chamber property'}
If we wish to apply Theorem \ref{thm:mpc}  for $q_{h,\cali}$, we need to find
a chamber associated with the denominator of the rational function $f(\bt_\cali)$, which contains all
vectors of type $b_k|_\cali$ where $b_k$ are the exponents appearing in the numerator.
The  next proposition  shows
 the existence of a chamber which contains the whole
projected real Lipman cone. In order to give some intuition for this fact, we also
provide an intermediate step of its proof.

For a  subset $\cali\subset \calv$, $\cali\not=\emptyset$ we define its
closure $\overline{\cali}$ as the
set of vertices of that {\it connected} minimal full subgraph $\Gamma_{\overline{\cali}}$ of $\Gamma$,
 which contains $\cali$.
We denote by $\delta_{v,\overline{\cali}}$ the valency of a vertex $v\in \overline{\cali}$ in the graph $\Gamma_{\overline{\cali}}$.

In \cite[Lemma 11]{LSz} is proved  that $f(\bt_\cali)$ has  a product decomposition of type
\begin{equation}\label{eq:fI}
f(\mathbf{t}_\cali) = R(\mathbf{t}_{\cali} ) \cdot \prod_{v\in \overline{\cali}}
 \Big( 1 -\mathbf{t}_{\cali}^{E^{*}_{v}} \Big)^{\delta_{v,\overline{\cali}}-2},
\end{equation}
where $R$ is a polynomial supported on $\pi_\cali(\calS')$, in particular it has no pole.
 Hence, the possible poles of $f$ via the $\cali$--reduction are drastically reduced from
the set of poles of  $\{1-\bt^{E^*_v}\}_{v\in \cale}$  to the set of poles of $\{1 -\mathbf{t}_{\cali}^{E^{*}_{v}}\}_{v\in \cale_{\overline{\cali}}}$.
 Here $\cale_{\overline{\cali}}$ is the set of end--vertices of $\Gamma_{\overline{\cali}}$;
 note that $\cale_{\overline{\cali}}\subset \cali$.
 Therefore, by the construction of the chamber decomposition of ${\mathbb R}^{|\cali|}$  (cf. \ref{sec:denom})
 the chambers associated with $f(\bt_\cali)$ are determined by the bases selected from
$ \{E^{*}_{v}|_\cali,E_u|_\cali :  v\in \mathcal{E}(\Gamma_{\overline{\mathcal{I}}}), u\in \mathcal{I}\}$.
These and a lattice--combinatorial argument provide
\begin{proposition}\label{Lm-6} \ \cite{LSz}
The interior of the projected  Lipman cone $\textnormal{int}\,  (\pi_{\mathcal{I}}(\mathcal{S}'_{\mathbb{R}}))$
 is contained entirely in a (big) chamber $\mathfrak{c}$ of $f(\bt_\cali)$.
\end{proposition}
\begin{remark} In \cite{LSz,LNN} is also proved the following `convexity property':
If  $l_0'\in Z_K-E+\calS'$ and  $[l'_0]=h$
then \ $q_{h,\cali} (l'_0)=q_{h,\overline{\cali}} (l'_0)$.
\end{remark}

\subsection{Duality for counting functions and periodic constants}
Now we are ready to prove our main theorem:
a duality/pairing between  special evaluations of the counting functions associated with $f$
and the periodic constants.
The duality is the upshot of {\em two `symmetries'},
manifested at two different levels. The
first one is the {\em equivariant Ehrhart--Macdonald--Stanley reciprocity
of the polytopes},
while the second is a topological imprint of the   Gorenstein duality
present at the level of the topological Poincar\'e series: a
$\{x\leftrightarrow Z_K-x\}$ symmetry.

\begin{theorem}\label{Qqduality} \ Fix  any
$\cali\subset \calv$, $\cali\not=\emptyset$. Then \\
\hspace*{2cm}
 (a) $\mathrm{mpc}^{\pi_{\mathcal{I}}(\mathcal{S}'_{\mathbb{R}})}
 (Z_h(\mathbf{t}_{\mathcal{I}}))=
 q_{[Z_K]-h,\mathcal{I}}(Z_K-r_h)$;\\
 \hspace*{2cm} (b) $\mathrm{pc}^{\pi_{\mathcal{I}}(\mathcal{S}'_{\mathbb{R}})}
 (Z_h(\mathbf{t}_{\mathcal{I}}))=
 Q_{[Z_K]-h,\mathcal{I}}(Z_K-r_h)$.

 \noindent
 That is, (in principle, the  hardly computable) periodic constant of a series $Z_h$
 can be determined as a precise finite coefficient counting of the dual series $Z_{[Z_K]-h}$.
\end{theorem}

\begin{proof}
By the inclusion--exclusion principle (\ref{eq:InEx}) $(b)$ is implied by  $(a)$.
Next we prove $(a)$.


The substitution $x\mapsto Z_K-x$ in  $f$, together with the
identities $Z_K-E=\sum_{v\in\mathcal{V}}(\delta_v-2)E^*_v$ from (\ref{ZK})
and $-2=\sum_{n\in \mathcal{V}}( \delta_n-2)$ (since $\Gamma$ is a tree)
 gives
\begin{equation}\label{eq:ZKsym}
f(\mathbf{t}_{\mathcal{I}})=\mathbf{t}_{\mathcal{I}}^{Z_K-E} \cdot f(\mathbf{t}_{\mathcal{I}}^{-1}).
\end{equation}

This on the Taylor expansion level transforms into the symmetry
 $T^{\infty}f(\bt_\cali)=\bt_{\cali}^{Z_K-E}  Z(\bt_\cali^{-1})=\sum_{l'\in \mathcal{S}'}z(l')\mathbf{t}_{\mathcal{I}}^{Z_K-E-l'}$.
The corresponding $h$-equivariant parts are
\begin{equation}\label{Zhdual}
(T^{\infty}f)_h(\bt_\cali)=\sum_{l'\in\mathcal{S}', [l']=[Z_K]-h} z(l')\mathbf{t}_{\mathcal{I}}^{Z_K-E-l'}.
\end{equation}

Furthermore,
(\ref{eq:fI}) and the sentence after it
shows that in the numerator of $f(\mathbf{t}_{\mathcal{I}})$ all the exponents are situated in the projected Lipman cone $\pi_{\mathcal{I}}(\mathcal{S}'_{\mathbb{R}})$. In particular, by Lemma \ref{Lm-6},
they are contained in a fixed chamber $\mathfrak{c}\subset {\mathbb R}^{|\cali|}$
which contains $\pi_{\mathcal{I}}(\mathcal{S}'_{\mathbb{R}})$.
Therefore Theorem \ref{thm:mpc} gives
$$\mathrm{mpc}^{\pi_{\mathcal{I}}(\mathcal{S}'_{\mathbb{R}})}
(Z_h(\mathbf{t}_{\mathcal{I}}))=\sum_{l'|_{\mathcal{I}}\leq (Z_K-E)|_{\mathcal{I}}, [l']=[Z_K]-h} z(l').$$
But the right--hand side
 is exactly the counting function $q_{[Z_K]-h,\mathcal{I}}(Z_K-r_h)$,
  since  $l'|_{\mathcal{I}}\leq (Z_K-E)|_{\mathcal{I}}$ is equivalent with $l'|_{\mathcal{I}}\prec (Z_K-r_h)|_{\mathcal{I}}$ if $[l']=[Z_K]-h$.
\end{proof}

\begin{corollary}\label{dualcor}  
$Q_{[Z_K]-h,\mathcal{N}}(Z_K-r_h)=-\mathfrak{sw}_h^{norm}(M)$; that is, the Seiberg--Witten invariant
can be expressed via the counting function as a finite sum of $Z$--coefficients.
\end{corollary}



\section{The `polynomial part' of the series $Z(\bt)$}\label{s:PP}

\subsection{\bf The `Polynomial--negative degree part' decomposition: motivation and history}\label{ss:PPmothist} Consider a one--variable
rational function $z(t)=B(t)/A(t)$
with $A(t)=\prod_{i=1}^n(1-t^{a_i})$ and $a_i>0$.
In \cite[7.0.2]{BN} is observed  that
any  such function
has a unique decomposition of the form $z(t)=P^+(t)+z^{neg}(t)$, where
$P^+(t)$ is a polynomial (with non-negative exponents) and $z^{neg}(t)$ is
a rational function of negative degree. Furthermore, $z(t)$ admits a periodic constant (associated with the Taylor expansion of $z$ and the cone $\mathbb{R}_{\geq0}$), which equals $P^+(1)$.
The decomposition can be established using the Euclidean division algorithm.
 $P^+$ is called the \emph{polynomial part} while the rational function $z^{neg}$ is the \emph{negative degree part} of the decomposition.

This one--variable
  periodic constant computation was useful in theorems of type
\ref{th:RedThm} with only one node (Seifert 3--manifolds), or surgery formulas
along one node, cf. \cite{BN}. Basically, in these cases, the concrete  computation of the periodic constant was based on the computation of $P^+$.

Though  the multivariable generalization was hardly needed, till now
the complete answer was not found. The main questions were: (a) what are the universal properties of the parts $P^+$ and $z^{neg}$, which guarantee
that a decomposition $z=P^++z^{neg} $ exists and it is unique; and,
(b) what algorithm provides this decomposition.
Additionally, the wished decomposition must satisfy (at least)
the next basic property: (c)
${\rm pc}(z)=P^+(1)$, where ${\rm pc}$ is associated with
 the Lipman cone (or, at least, with some subcone of it). (Hence,
 when all these are satisfied, whenever the Seiberg--Witten invariant is computable via such
periodic constant, e.g. when $\cali$ contains all the nodes, cf. Theorem \ref{th:RedThm}, or, the `bad' vertices of $\Gamma$,
cf. \cite{LNN},
then $P^+$ is a polynomial generalization of the Seiberg--Witten invariant.)

For functions with two variables
in \cite{LN}   a decomposition
was established satisfying all the required properties, based on a
`two--variable division procedure'.
Furthermore, for more variables, \cite{LSz} constructed  a candidate polynomial $P_1^+$
based on the combinatorics of $\Gamma$, Ehrhart theory and reduction to the
one- and two--variable divisions. It satisfied (c), but it didn't answer
(a) in a natural way.  Later,  \cite{LSzdiv}
 considered another  polynomial $P_2^+$ (with
 $P^+_2\not=P^+_1$ in general) constructed via an inductive
multivariable Euclidean division. It answered  (a)-(b), but
(c) was not established, so it was not clear if $P^+_2$  is helpful at all
in ${\rm pc}$ (or  Seiberg--Witten) computations.  The authors of
 \cite[4.4]{LSzdiv} not quite conjectured, but asked convincingly whether this $P_2^+$ is the right candidate for the polynomial part.
Below we show that this is indeed the case.

\subsection{\bf The polynomial part by division} \cite{LSzdiv}
We review in short
the needed statement of \cite{LSzdiv}  reorganized from the
point of view of the present note.
We start, similarly as in section \ref{s:EMS}, with a pair of free $\Z$--modules  $L\subset L'$, and a general
multivariable rational function (in variables $\bt^{L'}$)
\begin{equation*}
z(\bt)=\frac{\sum_{k=1}^r\iota_k\bt^{b_k}}{\prod_{i=1}^n
(1-\bt^{a_i})},
\end{equation*}
where  $\{\iota_k\}_{k=1}^r \in\mathbb{Z}$, $\{b_k\}_{k=1}^r, \ \{a_i\}_{i=1}^n\in L'$ such that $b_k\nprec 0$ for all $k$ and $0\prec a_i$
for all $i$. (In our application $z$ will be $f_h(\bt_\cali)$ for some $\cali\not=\emptyset$.)
\begin{proposition}\label{prop:LSzDIV} \ \cite{LSzdiv}\\
(a) $z(\bt)$ can be written in the following form by a `multivariable Euclidean division algorithm':
\begin{equation}\label{eq:EuAl}
z(\bt)=\sum_{S\subset \{1,\ldots, n\}}\
\frac{ \sum_j\, \iota_{S,j}\bt^{b_{S,j}}}
{\prod_{i\in S}(1-\bt^{a_i})},
\end{equation}
where  $\iota_{S,j}\in {\mathbb Z}$, $b_{S,j}\nprec 0$ for all $(S,j)$,
and $b_{S,j}\prec a_i$ for all $(S,j)$ and all $i\in S$ whenever $\calS\not=\emptyset$.

\noindent (b) \
$z(\bt)$ has a decomposition of type
$P^+(\bt)+ z^{neg}(\bt)$ with the next properties:

(i) \ $P^+(\bt)$ is a finite sum (polynomial)
$\sum_jn_j \bt^{c_j}$ with $c_j\nprec 0$ for all $j$;

(ii) \ $z^{neg}(\bt)$ is a rational function
 with negative degree in all variables $t_i$.

Furthermore, a decomposition of $z(\bt)$ with properties (i)-(ii) is unique.

\noindent (c) In fact, the terms  $P^+$ and $z^{neg}$ of the decomposition from (b)
are given via (\ref{eq:EuAl}) as follows:
 $P^+$ (resp. $z^{neg}$) is the sum of terms from (\ref{eq:EuAl})
 over $S=\emptyset$ (resp. $S\not=\emptyset$).
\end{proposition}
\begin{proof}
In {\it (a)} we use induction. Assume that we have an expression
$g(\bt)=\bt^b/ \prod_{i\in S}(1-\bt^{a_i})$ such that $b\nprec 0$ and
there exists some $i_0$ such that $b\nprec a_{i_0}$. Then we replace
$g$ by $-\bt^{b-a_{i_0}}/ \prod_{i\in S\setminus \{i_0\}}(1-\bt^{a_i})+
\bt^{b-a_{i_0}}/ \prod_{i\in S}(1-\bt^{a_i})$. Note that
$b-a_{i_0}\nprec 0$, hence the new fractions have similar form.
  Starting from the
original expression of $z$,  this step repeated whenever is applicable,
we get (\ref{eq:EuAl}).

{\it (b)} Define   $P^+$ and $z^{neg}$ as is indicated in (c). Then properties $(i)$--$(ii)$ are automatically satisfied. Next,  we prove the uniqueness  of the decomposition. We need to show that if $P^+(\bt)+ z^{neg}(\bt)=0$ ($\dag$)
and if $P^+$ and  $z^{neq}$ satisfy $(i)$--$(ii)$, then both are zero.
But, if $P^+$ is non--zero then by ($\dag$) and $(ii)$   $P^+$ has negative degrees
 in all variables, a fact which contradicts $(i)$.
\end{proof}

\subsection{\bf The polynomial part of $f_h(\bt_\cali)$ by duality}\label{ss:polpdual} \
Assume again that we are in the situation of a plumbing graph and its
rational function  $f(\bt)$ as in Sections \ref{s:prliminaries} and \ref{s:4}.

We fix $h\in H$. Then, by the proof of Theorem
\ref{Qqduality} we have
$f_h(\mathbf{t}_{\mathcal{I}})=\mathbf{t}_{\mathcal{I}}^{Z_K-E} \cdot f_{[Z_K]-h}(\mathbf{t}_{\mathcal{I}}^{-1})$ and
 $T^{\infty}f_h(\bt_\cali)=\bt_{\cali}^{Z_K-E}  Z_{[Z_K]-h}(\bt_\cali^{-1})$.
Write this expression
$\bt_{\cali}^{Z_K-E}  Z_{[Z_K]-h}(\bt_\cali^{-1})$ as $\sum _{l'\in Z_K-E-\calS'}w(l')\bt_\cali^{l'}$, where
$[l']=h$ automatically whenever $w(l')\not=0$.
define
\begin{equation}\label{eq:SUMf}
P^+_{h,\cali}(\bt_\cali):=\sum_{l'\nprec 0} w(l')\bt_\cali^{l'}, \ \ \ \
f^{neg}_{h,\cali}(\bt_\cali):= \sum_{l'\prec 0} w(l')\bt_\cali^{l'}.\end{equation}
Write also $P^+_{h,\cali}({\mathbf 1}):=
P^+_{h,\cali}(\bt_\cali)|_{t_i=1 \forall \, i}$.
\begin{theorem}\label{th:fneq} Consider the decomposition
$f_h(\bt_\cali)=P^+_{h,\cali}(\bt_\cali)+f^{neg}_{h,\cali}(\bt_\cali)$ from (\ref{eq:SUMf}).

(a) $P^+_{h,\cali}(\bt_\cali)$ and $f^{neg}_{h,\cali}(\bt_\cali)$ satisfy the requirements (i)--(ii) from Proposition \ref{prop:LSzDIV}(b).
In particular, by the uniqueness of the decomposition, this decomposition agrees with the decomposition
 from  \ref{prop:LSzDIV}(a)--(c) given by Euclidean division.

(b) $P^+_{h,\cali}({\mathbf 1})=
Q_{[Z_K]-h,\cali}(Z_K-r_h)$. In particular,
$P^+_{h,\cali}({\mathbf 1})=
{\rm pc} ^{\pi_\cali(\calS'_{{\mathbb R}})}(Z_h(\bt_\cali)).$
\end{theorem}
\begin{proof} {\it (a)} By its definition, $P^+_{h,\cali}(\bt_\cali)$
is a finite sum, hence
$f^{neg}_{h,\cali}(\bt_\cali)=
f_h(\bt_\cali)-P^+_{h,\cali}(\bt_\cali)$
is a rational function. Since in its expansion all monomials $w(l')\bt^{l'}$  satisfy
$l'\prec 0$, it has negative degree in all the variables.

{\it (b)} By the above transformations $w(l')=z(Z_K-E-l')$, $[l']=h$. Hence
$l'\nprec 0$ transforms into $Z_K-E\nprec Z_K-E-l'$, or
$Z_K-E-l'\not\geq Z_K-r_h$. This proves the  first identity. For
the second one use Theorem \ref{Qqduality} {\it (b)}.
\end{proof}
Corollary \ref{dualcor} and Theorem \ref{th:fneq} combined gives
\begin{corollary}\label{cor:Pegysw}
For any $h\in H$ one has
$P^+_{h,\caln}({\mathbf 1})=-\mathfrak{sw}_h^{norm}(M)$.
\end{corollary}

 We illustrate the above duality results on the following example.

\begin{example}
Let us consider the Brieskorn sphere $M=\Sigma(2,5,7)$. Recall that the
entries of  $E^*_v$'s  are the corresponding columns of the matrix
$ \{-(E^*_i,E^*_j)\}_{i,j}$, which is in fact $-I^{-1}$,
the negative of the inverse
of the intersection matrix. In this case the graph and $-I^{-1}$
 are the following:
\vspace{0.3cm}

\begin{picture}(200,50)(0,0)
\put(70,40){\makebox(0,0){\small{$-2$}}}
\put(70,50){\makebox(0,0){\small{$E_2$}}}
\put(90,40){\makebox(0,0){\small{$-1$}}}
\put(90,50){\makebox(0,0){\small{$E_1$}}}
\put(70,30){\circle*{3}}
\put(110,40){\makebox(0,0){\small{$-4$}}}
\put(110,50){\makebox(0,0){\small{$E_4$}}}
\put(130,40){\makebox(0,0){\small{$-2$}}}
\put(130,50){\makebox(0,0){\small{$E_5$}}}
\put(100,10){\makebox(0,0){\small{$E_3$}}}
\put(80,10){\makebox(0,0){\small{$-5$}}}
\put(90,30){\circle*{3}}
\put(110,30){\circle*{3}}
\put(130,30){\circle*{3}}
\put(90,10){\circle*{3}}
\put(70,30){\line(1,0){60}}\put(90,10){\line(0,1){20}}
\put(280,30){\makebox(0,0){
{$\tiny{
   -I^{-1}=
  \left[ {\begin{array}{ccccc}
   70 & 35 & 14 & 20 & 10 \\
   35 & 18 & 7 & 10 & 5 \\
   14 & 7 & 3 & 4 & 2 \\
   20 & 10 & 4 & 6 & 3 \\
   10 & 5 & 2 & 3 & 2 \\
  \end{array} } \right]}
$}
}}
\end{picture}

Therefore, the zeta-function reduced to the set of nodes $\caln=\{v_0\}$ and its `polynomial-negative degree part' decomposition (obtained by simple division
of polynomials) can be written as
$$f(t)=\frac{1-t^{70}}{(1-t^{35})(1-t^{14})(1-t^{10})}=
t+t^{11}+\frac{1-t+t^{15}+t^{21}}{(1-t^{14})(1-t^{10})}.$$
The same result follows by  duality as well.
Indeed, the Taylor expansion is  $Z(t)=1+t^{10}+t^{14}+\dots \ $ and one calculates $Z_K=(12,6,3,4,2)$ where we use short notation $(l_1,l_2,l_3,l_4,l_5)$ for $l=\sum_{i=1}^5 l_iE_i$. Hence, for
$t^{Z_K-E}Z(t^{-1})$ we get $t^{11}+t^1+t^{-3}+\cdots$, which gives
$P^+(t)=t^{11}+t$ again.
\end{example}

\subsection{Example. Plane curve singularities}
Let us describe the analogue of (\ref{eq:mpc}) and of the `polynomial-negative degree part'
decomposition for the embedded situation of an irreducible  plane curve singularity $g:(\mathbb{C}^2,0)\to (\mathbb{C},0)$.
In this case the series $z(t)$ is the Taylor expansion at the origin of the monodromy zeta function $\zeta(t)=\Delta(t)/(1-t)$,
cf. \cite{AC,AGZV}.
 A'Campo's formula constructs $\zeta(t)$ via  (\ref{eq:1.1}), but now applied to the
 minimal embedded resolution graph of the plane germ, whose unique $(-1)$-vertex has an extra arrowhead--neighbour (where  the arrow represents the strict transform of $\{g=0\}$, thus the $(-1)$-vertex is considered to be a node of the graph), and this series is reduced  to the variable of the $(-1)$-node.

 Let  $\mathcal{M}\subset \Z_{\geq 0}$ be the semigroup of $g$
consisting of all intersection multiplicities with all possible analytic germs.
 It is known \cite{cdg} that $z(t)=\sum_{s\in \mathcal{M}} t^s$. Let $\mu$ be the Milnor number of $g$.

We determine the polynomial part of $z(t)$  and we verify the analogue of (\ref{eq:mpc}). Here the needed Gorenstein duality reads as follows:
$(\dag)$ \ $s\not \in\mathcal{M}$ if and only if $\mu-1-s\in\mathcal{M}$.
 This also shows that $\Z\setminus \mathcal{M}$ is finite, its cardinality is exactly $\mu/2$, and the largest  element of $\Z\setminus \mathcal{M}$ is $\mu-1$.

Since $z(t)=\sum_{s\in\mathcal{M}}t^s=\sum_{s\geq 0}t^s-\sum_{s\not\in\mathcal{M}, s\geq 0}t^s$ we get that
$z^{neg}(t)=\sum_{s\geq 0}t^s=1/(1-t)$, and $P^+(t)=-\sum_{s\not\in\mathcal{M},s\geq 0}t^s$.
On the other hand, by duality $(\dag)$ one has
 $T^{\infty}z(t)=-t^{\mu-1}z(t^{-1})=-t^{\mu-1}(\sum_{s\in\mathcal{M}, s<\mu} t^{-s}+
\sum_{s\geq\mu} t^{-s})\stackrel{(\dag)}{=}-\sum_{s\not\in\mathcal{M},s\geq 0}t^s - (t^{-1}+t^{-2}+\cdots)
=P^+(t)- (t^{-1}+t^{-2}+\cdots)$,
whose part with positive exponents is exactly $P^+(t)$.

\section{Topological polytopes, counting lattice points and Seiberg--Witten invariants}\label{s:tPolSW}

\subsection{Motivation: polytopes and invariants of normal surface singularities}
The idea which connects geometric and topological invariants of normal surface singularities with the number of lattice points and the volume of certain polytopes has a long history
and it culminates in toric geometry. At the level of singularities,
the very first class  where this phenomenon appears is the Newton non-degenerate hypersurface
singularities (see eg. \cite{AGZV}). In this case one defines the Newton polytope $\calP$ associated with the defining equation and one proves that several invariants of the singularity can be recovered from $\calP$, see the following non-complete list of articles: \cite{Kouch,Var,BNnewt,D78,MT,Baldur}. In this article we would like to highlight only the geometric genus $p_g$ which equals the number of lattice points with strictly positive coordinates in $\calP$, cf.  \cite{MT}. Moreover, \cite{Baldur} shows that in the case when the germ has a rational homology sphere link $M$,  one has $p_g=-\frsw_0^{norm}(M)$ supporting the Seiberg--Witten invariant conjecture of N\'emethi and Nicolaescu \cite{NN1} for this case and, in particular, expressing the Seiberg--Witten invariant as a lattice point counting in the Newton polytope.

 This  analytic situation can be generalized to the case of
 isolated complete intersection singularities with certain non-degeneracy condition, cf. \cite{Khov,Mor,Okzeta,OkNdicis,BA}. For such germs, $p_g$ is expressed by \cite{Khov} and \cite{Mor} as an alternating sum of lattice points in certain polytopes obtained as
 Minkowski sums of Newton polyhedra defined by the equations of the germ.

 Having in mind the above correspondence $p_g=-\frsw_0^{norm}(M)$
 between the geometric genus and the Seiberg--Witten invariant
 (formulated by `Seiberg--Witten invariant conjecture' and
 valid for certain analytic types of singularities), it is natural to ask if there is an
 analogous interpretation of the  Seiberg--Witten invariant in terms of an alternating sum of
 lattice points in certain polytopes associated with the topological type.
 We will call such expression {\it inclusion--exclusion lattice point counting}, or IELP counting.

 The main result (and novelty)  of this section is
 a positive answer to this question, valid for any normal surface singularity with
 rational homology sphere link. Though in the case of $p_g$ and
 Merle--Teissier--Khovanskii--Morales type formulas the polytopes were
  associated with the equation of the germ,
 in this topological version they are associated with the resolution graph.

The wished expression starts with the results from
 \cite{LN}, which identifies the Seiberg--Witten invariant with certain coefficients of the Ehrhart quasipolynomial of a dilated polytope associated with the plumbing graph of the manifold $M$.
 However, merely with the techniques of \cite{LN} a concrete lattice point counting was obstructed.
 The point is that  the duality and construction of $P^+$ in the present note
 unlock the obstructions of \cite{LN} in the direction of an IELP counting.

\subsection{Preparation for an IELP expression.}\label{ss:preparation}
In section \ref{sec:denom} we have defined the linear map $\mathfrak{l}:\R^{|\cale|}\to \R^{|\calv|}$ by $\mathfrak{l}(\mathbf{x})=\sum_{v\in\cale}x_v E_v^*$ for any $\mathbf{x}=(x_v)_{v\in\cale}\in \R^{|\cale|}$ and we have considered the representation given by the composition
$\rho:\Z^{|\cale|}\stackrel{\mathfrak{l}|_{\Z^{|\cale|}}}{\longrightarrow} L'\longrightarrow H=L'/L.$
Set also the notation $\mathfrak{l}_\cali$ for the map defined by the composition of $\mathfrak{l}$ with the projection $\pi_{\cali}:\R^{|\calv|}\to \R^{|\cali|}$. In particular, for $\cali=\{v\}$ the maps $\mathfrak{l}_v$ are the components of $\mathfrak{l}$.

Recall also that for any $\cali\subset \calv, \cali\neq \emptyset$ we have defined the closed dilated
{\it convex polytopes}  $\calP_\cali^{(l')}$ indexed by $l'\in L'$:
 $\calP_\cali^{(l')}=\cap_{v\in \cali}\calP_v^{(l')}$, where
 $\calP_{v}^{(l')}=\{{\bf x}\in \mathbb{R}^{|\mathcal{E}|}_{\geq 0}\,:\,
\sum _{e\in \mathcal{E}}x_eE^*_e|_v\leq l'_v\}$.

In this section (similarly as in  \cite{LN}) we consider also the dilated {\it concave polytopes} as well:
\begin{equation}
\widetilde\calP_\cali^{(l')}=\bigcup_{v\in \cali} \calP_v^{(l')}\subset \R^{|\cale|}_{\geq 0}.
\end{equation}
>From the definition is clear that $\widetilde\calP_\cali^{(l')}$ depends only on the
restriction $l'|_{\cali}$. Furthermore, it can happen that in the union the contribution of some
$\calP_v^{(l')}$ is superfluous. This is what we clarify next.

If for some $\cali\subset \calv$ one has $l'=\sum _{v\in\cali}a_vE_v^*$ with all $a_v\not=0$,
then we say that the $E^*$--support of $l'$ is $\cali$, and we write $sp^*(l')=\cali$.
In the sequel we are interested mainly in those cases when the $E^*$--support of $l'$ is contained in
$\caln$. For such cycles we have the following additional `inclusion property' of the polytopes
$\widetilde\calP_\cali^{(l')}$.
\begin{lemma}\label{lem:red+}
Assume that $sp^*(l')=\cali\subset \caln$, $\cali\not=\emptyset$.
 Then $\widetilde\calP_{v}^{(l')}\subset \widetilde\calP_{\cali}^{(l')}$ for any $v\in\calv$.
\end{lemma}

\begin{proof}
The $\cali=\caln$ case appears in the proof of Theorem 5.4.2(a) of \cite{LN}. The general case is
rather similar, for the convenience of the reader we sketch the proof.
We assume that $v\not \in \cali$. Let $\Gamma_1$ be that connected component of the graph $\Gamma\setminus \cali$ which contains $v$. Then consider  those vertices  $\{u_k\}_{k\in\cali'}$ of $\cali$,
which are adjacent to $\Gamma_1$ ($\cali'\subset \cali$, $|\cali'|>0$). We claim that there exists positive rational numbers $\{r_k\}_{k\in \cali'}$ such that
\begin{equation}\label{eq:rk}
E^*_v-\textstyle{\sum_k} r_k \cdot E^*_{u_k}=i(E^*_v(\Gamma_1)),\end{equation}
where $i$ is the natural inclusion $L(\Gamma_1)\hookrightarrow L(\Gamma)$.
The numbers $r_k$ are determined in such a way that the cycle $D$ from the
left--hand side of (\ref{eq:rk})
has vanishing multiplicity along each $E_{u_k}$. This provides a linear system with $|\cali'|$ variables and equations with matrix $(E^*_{u_k}, E^*_{u_l})_{k,l}$ which is negative definite.
E.g., for further references, if $|\cali'|=1$ then
$r_1=(E^*_v,E^*_{u_1})/(E^*_{u_1},E^*_{u_1})$.

Once we have
determined $D$ one verifies (via intersection with base elements $E_v$) that it is supported on
$\Gamma_1$ and in fact equals with the cycle from the right--hand side of (\ref{eq:rk}).

(\ref{eq:rk}) shows that
$(\dag)$ $E^*_v=\sum_k r_k  E^*_{u_k}+\sum _{u\in\calv(\Gamma_1)} r_uE_u$ for certain
 positive rational numbers $\{r_{u}\}_{u\in\calv(\Gamma_1)}$.
This shows that for any $i\in \cali$ one has $(E^*_v,E^*_i)=\sum_kr_k (E^*_{u_k},E^*_i)$, hence
by the support assumption
$l'_v=\sum_k r_kl'_{u_k}$. Furthermore, the intersection of $(\dag)$ with $E^*_e$ for $e\in\cale$ gives
$E^*_e|_v\geq \sum_k r_k E^*_e|_{u_k}$. These facts prove that
$\widetilde\calP_{v}^{(l')}\subset \cup_k \widetilde\calP_{u_k}^{(l')}$.
\end{proof}

The duality results of the present article motivate us to distinguish a special element from $L'$
\begin{equation}\label{eq:ltop}
l'_{top}:=Z_K-E+\sum_{v\in\cale}E_v^*,
\end{equation}
which serves the role of a special parameter and defines the polytope 
$\widetilde\calP_\calv^{(l'_{top})}$.
By (\ref{ZK}) $l'_{top}=\sum_{v\in\caln}(\delta_v-2)E_v^*$ too, hence $sp^*(l_{top}')=\caln$ and by
Lemma \ref{lem:red+} we obtain
$\widetilde\calP_\calv^{(l'_{top})}=\widetilde\calP_\caln^{(l'_{top})}$.
We call it the  {\em topological polytope} associated with the graph $\Gamma$
and we denote it by $\widetilde\calP_\caln^{\,top}$.

A series of sub--polytopes of $\widetilde\calP_\caln^{(l'_{top})}$ are defined as follows.
Let us write the set of nodes as $\caln=\{v_1, \ldots, v_s\}$. Then
consider the multiset $\caln^m:=\{v_1^{\delta_{v_1}-2},\ldots, v_s^{\delta_{v_s}-2}\}$, where in
$\caln^m$ each node $v$ has multiplicity $\delta_{v}-2$. Since $\sum_{v\in \caln}(\delta_v-2)
=|\cale|-2$, $\caln^m$ contains $|\cale|-2 $ symbols.
Let $\cali^m\subset \caln^m$ be a non--trivial sub--multiset, $\cali^m=\{v_{i_1}^{k_{i_1}},\ldots, v_{i_r}^{k_{i_r}}\}$,
where $\{v_{i_1},\ldots, v_{i_r}\}=\cali\subset \caln$ and
$0<k_{i_j}\leq \delta_{v_{i_j}}-2$ for all $j$, and  $0<r\leq s$. To such $\cali^m$ we define $|\cali^m|:=\sum_j k_{i_j}$, the cycle
$l'(\cali^m)=\sum_j k_{i_j}E^*_{v_{i_j}}$ with $E^*$--support $\cali$,  and the polytope
$\widetilde\calP_{\cali^m}:= \widetilde{\calP}_\cali ^{\,l'(\cali^m)}$.

\begin{theorem}\label{th:swform}
For any dilated polytope $\calP_\cali^{(l')}$ we denote by $R_h(\calP_\cali^{(l')})$ the number of integral points in $\calP_\cali^{(l')}\cap \rho^{-1}([l']-h)$ with all coordinates strictly positive. Then
\begin{equation}\label{eq:sw}
-\mathfrak{sw}_h^{norm}(M)=\sum_{\emptyset\neq
\cali^m\subset \caln^m}(-1)^{|\cale|-|\cali^m|}R_h(\widetilde\calP_{\cali^m}).
\end{equation}
\end{theorem}

\begin{proof}We need to compute $\calP^+_{h,\caln}({\bf 1})$. Let us fix some $\cali\subset \calv$,
and we reformulate $\calP^+_{h,\cali}({\bt_\cali})$.

 Sometimes is more convenient to calculate $\calP^+_{h,\cali}({\bt_\cali})$
 via the truncation of the series $Z_{[Z_K]-h}(\bt_\cali)$,
  the  {\em `dual' polynomial part}
\begin{equation}\label{eq:dualpol0}
\check{P}_{h,\cali}^+(\bt_\cali):=\sum_{[l']=[Z_K]-h\atop l'\nsucc Z_K-E} z(l')\bt_\cali^{l'}
\end{equation}
using the defining identity  (\ref{eq:SUMf})
\begin{equation}\label{eq:dual}
P_{h,\cali}^+(\bt_\cali)=\bt_\cali^{Z_K-E} \check{P}_{h,\cali}^+(\bt_\cali^{-1}).
\end{equation}
Then (\ref{eq:dualpol0}) and (\ref{eq:1.1}) gives for the dual polynomial part
\begin{equation}\label{eq:dualpol}
\check{P}_{h,\cali}^+(\bt_\cali)=\sum (-1)^{\sum_{v\in\caln}k_v}\cdot\binom{\delta-2}{k} \bt_\cali^{\sum_{v\in\caln}k_vE^*_v + \mathfrak{l}(\mathbf{x})}
\end{equation}
where $\binom{\delta-2}{k}:=\prod_{v\in\caln}\binom{\delta_v-2}{k_v}$ and the sum runs over all $0\leq k_v\leq \delta_v-2$ for every $v\in\caln$ and $\mathbf{x}\in \Z_{\geq0}^{|\cale|}\cap \rho^{-1}([Z_K-\sum_{v\in\caln}k_vE^*_v]-h)$ such that $\mathfrak{l}(\mathbf{x})\nsucc Z_K-E-\sum_{v\in\caln}k_vE^*_v$.

Now apply (\ref{eq:dual}),  (\ref{ZK}) and replace $k_v$ by $\delta_v-2-k_v$ in order to get  for the polynomial part
\begin{equation}\label{eq:pol}
P_{h,\cali}^+(\bt_\cali)=(-1)^{|\cale|}\sum (-1)^{\sum_{v\in \caln} k_v} \binom{\delta-2}{k}\bt_\cali^{\sum_\caln k_v E^*_v-\mathfrak{l}(\mathbf{x})},
\end{equation}
where the sum runs over all $0\leq k_v\leq \delta_v-2$ and $\mathbf{x}\in \widetilde\calP_\calv^{(\sum_{v\in\caln} k_v E^*_v)}\cap \rho^{-1}([\sum_{v\in\caln} k_v E^*_v]-h)\cap\Z_{>0}^{|\cale|}$.

If we specify $\cali=\caln$, Lemma \ref{lem:red+} implies that the sum in (\ref{eq:pol}) runs over all
$\cali^m\subset \caln^m$, $\cali^m\not=\emptyset$, and
lattice points of $\widetilde{P}_{\cali^m}\cap \rho^{-1}([\sum_I k_v E^*_v]-h)$ with all coordinates being strictly positive. Now, by Corollary \ref{cor:Pegysw} the formula (\ref{eq:sw}) follows.
\end{proof}

\begin{remark}
In the above proof the expressions with
the binomial coefficients coming from the numerator of $f(\bt)$
have appeared rather naturally. However, we prefer the multiset language of (\ref{eq:sw}) too,
since provides a different interpretation for the appearance of the binomial coefficients,
 and furthermore, in this way we obtain a perfect similarity
with the  geometric genus formula of Khovanskii \cite{Khov} and Morales \cite{Mor} for Newton non-degenerate isolated complete intersection singularities
 (supporting the intimate relationship between these two invariants).
\end{remark}

\begin{corollary}\label{prop:latcount}
If the conditions $Z_K|_{\caln}\leq E_v^*|_{\caln}$ for every node $v\in\caln$ are satisfied then the polynomial part can be expressed as
$$P_{h,\caln}^+(\mathbf{t}_{\caln})=\sum_{\mathbf{x}\in (\widetilde\calP_\caln^{top}
\setminus \widetilde\calt_\caln^{top})\cap \rho^{-1}([Z_K]-h)}\mathbf{t}_{\caln}^{Z_K-E-\mathfrak{l}(\mathbf{x})},$$
where $\widetilde\calt_{\caln}^{top} = \widetilde{\mathcal{F}}_{\caln}^{top} \setminus \widetilde{\mathcal{G}}^{co}_{\caln}$ is the set of `non-coordinate facets' of
$\widetilde\calP_\caln^{top}$. In  particular,
$$-\mathfrak{sw}^{norm}_h(M)= \mbox{cardinality of}\ \ \big((\widetilde\calP_\caln^{top}\setminus \widetilde\calt_\caln^{top})\cap \rho^{-1}([Z_K]-h)\big).$$
\end{corollary}

\begin{proof}
In (\ref{eq:pol}) only the submultiset $\cali^m=\caln^m$ has non--trivial contribution.
\end{proof}

\begin{example}
Let $M$ be a negative definite Seifert fibered rational homology sphere, ie. $\Gamma$ is a star-shaped tree. Assume that the central node $v_0$ has $d\geq 3$ legs, each with determinant $\alpha_i\geq 2$ ($i=1,\ldots, d$), and let $e$ be the orbifold Euler number of $M$. Then by the formula (\ref{ZK}) one calculates $(Z_K-E)|_{v_0}=(d-2-\sum_i 1/\alpha_i)/|e|$; and one also has $E_{v_0}^*|_{v_0}=1/|e|$
 (cf.  \cite[6.1.3]{LN}, \cite[11.1]{NOSZ}). Hence
 the assumption of Corollary \ref{prop:latcount} reads as
$$d-3<\sum_{i=1}^d\frac{1}{\alpha_i}.$$
Notice that this cannot happen if $d\geq 6$ and it is always satisfied whenever $d=3$. Hence, in particular, Proposition \ref{prop:latcount} implies a well--known result in the literature for Brieskorn homology spheres $\Sigma(p,q,r)$, saying that the normalized Seiberg--Witten invariant, or, equivalently, the normalized Casson invariant,  can be expressed as the number of lattice points in the polytope
$$\calP_{v_0}=\{\mathbf{x}\in \R_{\geq 0}^3 \ : \ pq\cdot x_1+qr\cdot x_2+pr \cdot x_3\leq pqr\}$$
but not on the non-coordinate facet. In fact, in this case $\calP_{v_0}$ coincides with the Newton polytope associated with the equation of the Brieskorn germ $(\{x^p+y^q+z^r=0\},0)$.

More details and different aspects can be found in \cite{FS},\cite{Nic01}, \cite{Sav}, \cite{NN1}, \cite{CK} and \cite{LN}.
\end{example}

\begin{example}\label{ex:tPolSW}
Let $M$ be the manifold associated with the following plumbing graph $\Gamma$:

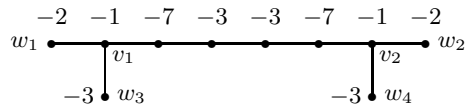
\begin{picture}(200,50)(-60,0)
\put(70,40){\makebox(0,0){\small{$-2$}}}
\put(60,30){\makebox(0,0){\small{$w_1$}}}\put(220,30){\makebox(0,0){\small{$w_2$}}}
\put(90,40){\makebox(0,0){\small{$-1$}}}
\put(97,25){\makebox(0,0){\small{$v_1$}}}\put(197,25){\makebox(0,0){\small{$v_2$}}}
\put(70,30){\circle*{3}}
\put(110,40){\makebox(0,0){\small{$-7$}}}
\put(130,40){\makebox(0,0){\small{$-3$}}}
\put(100,10){\makebox(0,0){\small{$w_3$}}}\put(200,10){\makebox(0,0){\small{$w_4$}}}
\put(80,10){\makebox(0,0){\small{$-3$}}}
\put(180,10){\makebox(0,0){\small{$-3$}}}
\put(150,40){\makebox(0,0){\small{$-3$}}}

\put(170,40){\makebox(0,0){\small{$-7$}}}

\put(190,40){\makebox(0,0){\small{$-1$}}}
\put(210,40){\makebox(0,0){\small{$-2$}}}

\put(90,30){\circle*{3}}
\put(110,30){\circle*{3}}
\put(130,30){\circle*{3}}
\put(90,10){\circle*{3}}
\put(150,30){\circle*{3}}
\put(170,30){\circle*{3}}
\put(190,30){\circle*{3}}
\put(210,30){\circle*{3}}
\put(190,10){\circle*{3}}

\put(70,30){\line(1,0){140}}\put(90,10){\line(0,1){20}}
\put(190,10){\line(0,1){20}}
\end{picture}

$M$ is realized e.g. as the link
 of the Newton non-degenerate hypersurface singularity  $f:(\C^3,0)\to (\C,0)$, $f(x,y,z)=x^{13}+y^{13}+z^{3}+x^2y^2$. In the sequel, following the above discussions, we present the computation of the (dual) polynomial part using the lattice points of the topological polytope.

>From the negative inverse of the intersection form we read the linear forms
$$\mathfrak{l}_{v_j}:\R^4=\R\langle E_{w_1},\dots,E_{w_4}\rangle\to \R: \ \
\mathfrak{l}_{v_1}(\mathbf{x})=\langle(33,6,22,4),\mathbf{x}\rangle \ \mbox{and} \ \mathfrak{l}_{v_2}(\mathbf{x})=\langle(6,33,4,22),\mathbf{x}\rangle$$ associated with the nodes $v_1$ and $v_2$. Moreover, one gets $l'_{top}|_{\caln}=(78,78)$ and $Z_K|_{\caln}=(14,14)$ in the $(E_{v_1},E_{v_2})$-basis, hence the topological polytope is given by
$$\widetilde\calP^{top}_{\caln}:=\{\mathbf{x}\in \R^4 \ : \ \langle(33,6,22,4),\mathbf{x}\rangle\leq 78 \ \ \mbox{or} \ \ \langle(6,33,4,22),\mathbf{x}\rangle\leq 78\}.$$
The group $H\simeq\Z_3$ is generated by eg. $h_1=[E^*_{w_3}]$ and one can see by calculations that $[Z_K]=0$, or by the knowledge that $M$ is the link of a Gorenstein (in this case hypersurface) singularity. Moreover, $[E^*_{v_1}]=[E^*_{v_2}]=0$ too. Then (\ref{eq:dualpol}) implies that
$$\check{P}_{h,\caln}^+(\bt)=\sum_{\star}\bt^{\mathfrak{l}(x)}-\sum_{\star_1}\bt^{E^*_{v_1}+\mathfrak{l}(x)}-\sum_{\star_2}\bt^{E^*_{v_2}+\mathfrak{l}(x)},$$
where $\bt:=\bt_\caln=(t_{v_1},t_{v_2})$ and the sums run under the following conditions:
\begin{eqnarray*}
(\star) \ : \ \ &\mathbf{x}\in\Z^4\cap \rho^{-1}(-h) \ &\mbox{and} \ \ \mathfrak{l}_\caln(\mathbf{x})\nsucc (13,13);\\
(\star_1) \ : \ \ &\mathbf{x}\in\Z^4\cap \rho^{-1}(-h) \ &\mbox{and} \ \ \mathfrak{l}_\caln(\mathbf{x})\nsucc (-53,1);\\
(\star_2) \ : \ \ &\mathbf{x}\in\Z^4\cap \rho^{-1}(-h) \ &\mbox{and} \ \ \mathfrak{l}_\caln(\mathbf{x})\nsucc (1,-53).
\end{eqnarray*}
Then it can be checked that for $h=0$ the lattice points satisfying the condition $(\star)$ are $(0,0,0,0)$, $(1,0,0,0)$, $(0,1,0,0)$, $(2,0,0,0)$, $(0,2,0,0)$, $(0,0,3,0)$ and $(0,0,0,3)$, creating the monomials $1$, $\bt^{(33,6)}$, $\bt^{(6,33)}$, $\bt^{(66,12)}$, $\bt^{(12,66)}$, $\bt^{(66,12)}$ and $\bt^{(12,66)}$ respectively. For $(\star_1)$ and $(\star_2)$ the only lattice point is $(0,0,0,0)$ with the associated monomials $\bt^{(66,12)}$ and $\bt^{(12,66)}$. Therefore, $\check{P}_{0,\caln}^+(\bt)=1+\bt^{(33,6)}+\bt^{(6,33)}+\bt^{(66,12)}+\bt^{(12,66)}$ which by (\ref{eq:dual}) implies
$$P_{0,\caln}^+(\bt)=\bt^{(13,13)}\cdot \check{P}_{h,\caln}^+(\bt^{-1})=\bt^{(1,-53)}+\bt^{(-53,1)}+\bt^{(7,-20)}+\bt^{(-20,7)}+\bt^{(13,13)}\ \mbox{and } \ -\mathfrak{sw}^{norm}_0(M)=5.$$
Similarly, for $h_1$ we get the lattice points (from $\rho^{-1}(2h_1)$) $(0,0,0,1)$, $(0,0,2,0)$ and $(0,1,0,1)$ satisfying $(\star)$ and associating the monomials $\bt^{(4,22)}$, $\bt^{(44,8)}$ and $\bt^{(10,55)}$. Hence, $\check{P}_{h_1,\caln}^+(\bt)=\bt^{(4,22)}+\bt^{(44,8)}+\bt^{(10,55)}$ and $-\mathfrak{sw}^{norm}_{h_1}(M)=3$. Symmetrically, one has $\check{P}_{2h_1,\caln}^+(\bt)=\bt^{(22,4)}+\bt^{(8,44)}+\bt^{(55,10)}$ and $-\mathfrak{sw}^{norm}_{2h_1}(M)=3$.

\end{example}

\subsection{Minkowski sums of polyhedrons associated with nodes}

For all the polytopes $\calP$ considered in this subsection we denote by $\calP^{\triangleleft}:=\calP\setminus (\calF\setminus\mathcal{G}^{co})$ the corresponding semi-open polytope and define the unbounded convex polyhedron
$\calP^+:=\R_{\geq 0}^{|\cale|}\setminus\calP^{\triangleleft}$.
In the sequel we discuss relations between polyhedrons of type $\widetilde{\calP}_{\cali^m}^+$ and Minkowski sums of $(\calP_{v}^{(E^*_v)})^+$. In particular, we prove for special cases a Khovanskii--Morales type IELP expression for the normalized Seiberg--Witten invariants using Minkowski sums.

\begin{lemma}\label{lem:vertpol}
%
For any $v\in \caln$ let $\cale_v$ be the set of those end-vertices $w\in\cale$ for which the unique minimal connected full subgraph $[w,v]$ does not contain any other nodes than $v$. Then for any $v'\in \caln$ we have
\begin{equation}
\frac{E^*_w|_v}{E^*_w|_{v'}}=\frac{E^*_v|_v}{E^*_{v'}|_{v}}=:\lambda_{vv'} \ \ \ \ \mbox{for any} \ \ w\in\cale_v.
\end{equation}
Moreover, if $v''\in\caln$ is a vertex of the subgraph $[v,v']$ for some $v,v'\in\caln$ then $\lambda_{vv'}=\lambda_{vv''}\cdot\lambda_{v''v'}$.
\end{lemma}

\begin{proof}
It follows from
the formula of \cite{EN} expressing $-(E^*_{v'},E^*_{v''})$ as the fraction of the determinant (of the negative intersection form) of the graph $\Gamma\setminus[v',v'']$ and $|H|$. Here, $[v',v'']$ denotes the unique minimal connected subgraph containing the vertices $v'$ and $v''$.
%
\end{proof}

\begin{prop}\label{prop:MS}
(a) For any $k_v\in\Z_{\geq 0}, v\in\caln$ write $l':=\sum_vk_vE^*_v$. Then
 $$\sum_{v\in \caln}k_v (\calP^{(E^*_v)} _v)^{+} \subset
 ( \widetilde{\calP}^{(l')}_{sp^*(l')})^{+}.$$

(b) Let $v_1, v_2 \in\caln$ be two neighbouring nodes in $\Gamma$, ie. the minimal connected subgraph $[v_1,v_2]$ does not contain any other nodes. Then
$k_{v_1} (\calP^{(E^*_{v_1})}_{v_1})^{+}+k_{v_2} (\calP^{(E^*_{v_2})}_{v_2})^{+}=
(\widetilde{\calP}^{(k_{v_1}E^*_{v_1}+ k_{v_2}E^*_{v_2})}_{\{v_1,v_2\}})^{+}$.
\end{prop}

\begin{proof}
(a) By definition of $(\calP^{(E^*_v)}_v)^+$ follows that the Minkowski sum $\sum_{v\in\caln}k_v(\calP^{(E^*_v)}_v)^+$ is also a convex polyhedron bounded from below by the convex hull of its vertices and the coordinate hyperplanes. Any vertex $\mathbf{v}$ of the Minkowski sum can be written as $\sum_{v\in\caln}\sum_{j=1}^{k_v}\mathbf{v}^j_v$, where $\mathbf{v}^j_v$ are vertices of the polyhedron $(\calP^{(E^*_v)}_v)^+$. Then we can write $\mathfrak{l}(\mathbf{v}^j_v)=r_v^{-1}E^*_w$ for some $w\in\cale$ and $r_v\in\Q_{>0}$ as in the proof of Lemma \ref{lem:red+}. Therefore we have $\mathfrak{l}(\mathbf{v}^j_v)\geq E^*_v$ which implies the inequality $\mathfrak{l}_{\caln}(\mathbf{v}_{\mathbf{k}})\geq \sum_{v\in\caln}k_vE^*_v|_{\caln}$ for all vertices, hence for all $\sum_{v\in\caln}k_v(\calP^{(E^*_v)}_v)^+$ too.

(b) Since the coefficients $k_{v_1},k_{v_2}$ do not change the combinatorial type of neither of the two polyhedrons (see eg. \cite{Gstrum} for Minkowski sums) we assume that $k_{v_1}=k_{v_2}=1$. By (a) we have to prove that the compact faces of
$(\widetilde{\calP}^{(E^*_{v_1}+ E^*_{v_2})}_{\{v_1,v_2\}})^{+}$
can be written as Minkowski sums of compact faces of $(\calP^{(E^*_{v_1})}_{v_1})^{+}$ and $(\calP^{(E^*_{v_2})}_{v_2})^{+}$. By the definition of
$\widetilde{\calP}^{(E^*_{v_1}+ E^*_{v_2})}_{\{v_1,v_2\}}$
these compact faces are given by the hyperplanes $\mathfrak{l}_{v_1}=(E^*_{v_1}+E^*_{v_2})|_{v_1}$,  $\mathfrak{l}_{v_2}=(E^*_{v_1}+E^*_{v_2})|_{v_2}$ and their intersections. Now, consider the stratification $\cale=\cale_1\amalg\cale_2$ of the set of ends so that $\cale_i$ are those end-vertices which are contained in the same connected subgraph as $v_i$ if we cut out $[v_1,v_2]\setminus\{v_1,v_2\}$ from $\Gamma$. We denote by $\mathbf{v}^{(j)}_{i}$ the (non-zero) vertex of $\calP^*_{v_i}$ which are situated on the coordinate axis corresponding to $w\in\cale_j$ and let $F\langle\mathbf{v}^{(j)}_{i,w}\rangle$ be the face generated by the vertices of type $\mathbf{v}^{(j)}_{i,w}$ for some fixed $i$ and $j$.

Then our claim is that the compact facets of
$\widetilde{\calP}^{(E^*_{v_1}+ E^*_{v_2})}_{\{v_1,v_2\}}$
are given by the Minkowski sums $F\langle\mathbf{v}^{(1)}_{1},\mathbf{v}^{(2)}_{1}\rangle+F\langle\mathbf{v}^{(2)}_{2}\rangle$, $F\langle\mathbf{v}^{(1)}_{1}\rangle+F\langle\mathbf{v}^{(2)}_{1}\mathbf{v}^{(2)}_{2}\rangle$, and their intersection is given by $F\langle\mathbf{v}^{(1)}_{1}\rangle+F\langle\mathbf{v}^{(2)}_{2}\rangle$. Indeed, using Lemma \ref{lem:vertpol} we deduce that
$$\mathfrak{l}_{v_1}=\mathfrak{l}_{v_1}|_{\cale_{1}}+(1/\lambda_{v_2v_1})\mathfrak{l}_{v_2}|_{\cale_{2}} \ \ \ \mbox{and } \ \ \  \mathfrak{l}_{v_2}=(1/\lambda_{v_1v_2})\mathfrak{l}_{v_1}|_{\cale_{1}}+\mathfrak{l}_{v_2}|_{\cale_{2}},$$
where $\mathfrak{l}_{v_i}|_{\cale_{i}}(\mathbf{x})=\sum_{w\in\cale_i}x_w E^*_w|_{v_i}$. Hence, $\mathfrak{l}_{v_1}(F\langle\mathbf{v}^{(1)}_{1},\mathbf{v}^{(2)}_{1}\rangle+F\langle\mathbf{v}^{(2)}_{2}\rangle)=E^*_{v_1}|_{v_1}+(1/\lambda_{21})E^*_{v_2}|_{v_2}=(E^*_{v_1}+E^*_{v_2})|_{v_1}$. Similarly, one has
$\mathfrak{l}_{v_2}(F\langle\mathbf{v}^{(1)}_{1}\rangle+F\langle,\mathbf{v}^{(2)}_{1}\mathbf{v}^{(2)}_{2}\rangle)=(E^*_{v_1}+E^*_{v_2})|_{v_2}$ and $F\langle\mathbf{v}^{(1)}_{1}\rangle+F\langle\mathbf{v}^{(2)}_{2}\rangle$ satisfies both equations.
\end{proof}

\begin{corollary}\label{cor:swMS}
For a graph $\Gamma$ with two nodes the formula (\ref{eq:sw}) can be interpreted by Minkowski sums in the following way

\begin{equation*}\label{eq:swMin}
-\mathfrak{sw}_h^{norm}(M)=\sum_{\emptyset\neq
\cali^m\subset \caln^m}(-1)^{|\cale|-|\cali^m|}
\bar{R}_h\Big(\sum_{j}k_{i_j}(\calP_{v_{i_j}}^{(E^*_{v_{i_j}})})^+\Big),
\end{equation*}
where $\bar{R}_h(\calP)$ is a $\rho^{-1}([l']-h)$-lattice point counting with all coordinates being strictly positive and not in the interior of $\calP$.
\end{corollary}
The next example has two messages. First, it is an example with three nodes when
the corresponding identity from Proposition \ref{prop:MS}(b) is not valid.
On the other hand, for this example the number of lattice points in the two (distinct)
polytopes agree. This fact suggests  that though
 the validity of Proposition \ref{prop:MS}(b) cannot be expected in general, a Khovanskii--Morales type IELP  formula for the normalized Seiberg--Witten invariants using Minkowski sums may exist.
\begin{example}
Consider the following graph with trivial $H$:

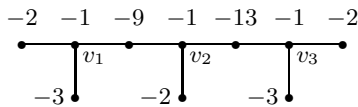
\begin{picture}(200,50)(-65,0)
\put(70,40){\makebox(0,0){\small{$-2$}}}
\put(90,40){\makebox(0,0){\small{$-1$}}}
\put(97,25){\makebox(0,0){\small{$v_1$}}}
\put(177,25){\makebox(0,0){\small{$v_3$}}}
\put(137,25){\makebox(0,0){\small{$v_2$}}}
\put(70,30){\circle*{3}}
\put(110,40){\makebox(0,0){\small{$-9$}}}
\put(130,40){\makebox(0,0){\small{$-1$}}}
\put(80,10){\makebox(0,0){\small{$-3$}}}\put(120,10){\makebox(0,0){\small{$-2$}}}
\put(160,10){\makebox(0,0){\small{$-3$}}}
\put(150,40){\makebox(0,0){\small{$-13$}}}

\put(170,40){\makebox(0,0){\small{$-1$}}}

\put(190,40){\makebox(0,0){\small{$-2$}}}

\put(90,30){\circle*{3}}
\put(110,30){\circle*{3}}
\put(130,30){\circle*{3}}
\put(130,10){\circle*{3}}
\put(90,10){\circle*{3}}
\put(150,30){\circle*{3}}
\put(170,30){\circle*{3}}
\put(190,30){\circle*{3}}

\put(170,10){\circle*{3}}

\put(70,30){\line(1,0){120}}\put(90,10){\line(0,1){20}}
\put(170,10){\line(0,1){20}}\put(130,10){\line(0,1){20}}
\end{picture}

The normalized Seiberg--Witten invariant has been calculated in \cite[6.5]{LSz} using the polynomial $P^+_1$ (see \ref{ss:PPmothist} for details) developed in the very same article. Now we restrict our attention to the Minkowski sums of polyhedrons $(\calP^{(E^*_{v_i})}_{v_i})^+$.

The polytopes associated with the nodes are $\calP^{(E^*_{v_1})}_{v_1}=\{\mathbf{x}\in \R^5 : \langle(93,62,42,36,24),\mathbf{x}\rangle\leq 186\}$, $\calP^{(E^*_{v_2})}_{v_2}=\{\mathbf{x}\in \R^5 : \langle(42,28,21,18,12),\mathbf{x}\rangle\leq 42\}$ and $\calP^{(E^*_{v_3})}_{v_3}=\{\mathbf{x}\in \R^5 : \langle(36,24,18,21,14),\mathbf{x}\rangle\leq 42\}$. The Minkowski sums can be described as follows: $(\calP^{(E^*_{v_1})}_{v_1})^+ + (\calP^{(E^*_{v_2})}_{v_2})^+=(\widetilde{\calP}^{(E^*_{v_1}+E^*_{v_2})}_{\{v_1,v_2\}})^+$, $(\calP^{(E^*_{v_2})}_{v_2})^+ + (\calP^{(E^*_{v_3})}_{v_3})^+=(\widetilde{\calP}^{(E^*_{v_2}+E^*_{v_3})}_{\{v_2,v_3\}})^+$ follows from Proposition \ref{prop:MS}(b), while
$$(\calP^{(E^*_{v_1})}_{v_1})^+ + (\calP^{(E^*_{v_3})}_{v_3})^+=(\widetilde{\calP}^{(E^*_{v_1}+E^*_{v_3})}_{\{v_1,v_3\}}\cup \calP^{(l')}_{v_1 v_3})^+$$
where the extra polytope is given by
$$\calP^{(l')}_{v_1 v_3}:=\{\mathbf{x}\in \R^5 : \langle(93,62,42,49,98/3),\mathbf{x}\rangle\leq l'|_{v_1}:=(E^*_{v_1}+(\lambda_{v_2v_3}\lambda_{v_3v_1}/\lambda_{v_2v_1})\cdot E^*_{v_3})|_{v_1}=284\}$$ and $\lambda_{v_iv_j}$ are the constants from Lemma \ref{lem:vertpol}.  Similarly, one has
$$(\calP^{(E^*_{v_1})}_{v_1})^+ + (\calP^{(E^*_{v_2})}_{v_2})^+ +(\calP^{(E^*_{v_3})}_{v_3})^+=(\widetilde{\calP}^{(E^*_{v_1}+E^*_{v_2}+E^*_{v_3})}_{\{v_1,v_2,v_3\}}\cup \calP^{(l'+E^*_{v_2})}_{v_1 v_3})^+.$$
One can calculate $R_0(\widetilde{\calP}_{\{v_i\}})=R_0(\widetilde{\calP}_{\{v_2,v_3\}})=0$, $R_0(\widetilde{\calP}_{\{v_1,v_2\}})=1$, and we also find the identities  $\bar{R}_0((\calP^{(E^*_{v_1})}_{v_1})^+ + (\calP^{(E^*_{v_3})}_{v_3})^+)=R_0(\widetilde{\calP}_{\{v_1,v_3\}})=1$ and
$\bar{R}_0((\sum_{i=1}^3\calP^{(E^*_{v_i})}_{v_i})^+)=R_0(\widetilde{\calP}_{\{v_1,v_2,v_3\}})=15$. Hence,
$$-\mathfrak{sw}_h^{norm}(M):=\sum_{\emptyset\neq I\subset \{v_1,v_2,v_3\}}(-1)^{5-|I|}\bar{R}_0(\sum_{i\in I}(\calP^{(E^*_{v_i})}_{v_i})^+)=13.$$

\end{example}

%
%


\begin{thebibliography}{30}

\bibitem[A'C75]{AC} A'Campo, N.: La fonction z\^eta d'une monodromie,
Comment. Math. Helv. {\bf 50} (1975), 233-248.

\bibitem[AGZV12]{AGZV} Arnold V.I., Gusein-Zade S.M. and Varchenko A.N.: Singularities of differentiable maps, Vol. 2. Monodromy and asymptotics of integrals, {\em Modern Birkh\"auser Classics, Birkh\"auser/Springer}, New York (2012). Translated from the Russian by Hugh Porteous.

%
%

\bibitem[BA07]{BA} Bivi\`a-Ausina, C.: Mixed Newton numbers and isolated complete intersection singularities, {\em Proc. Lond.
Math. Soc.} {\bf 94} (3) (2007), 749--771.

\bibitem[BN10]{BN} Braun, G. and N\'emethi, A.:
Surgery formula for Seiberg--Witten invariants of negative definite plumbed 3--manifolds,
{ \em J. f\"ur die reine und ang. Math.} {\bf 638} (2010), 189--208.

\bibitem[BN07]{BNnewt} Braun, G. and N\'emethi, A.:
Invariants of Newton non-degenerate surface singularities,
{\em Compositio Math.} {\bf 143} (2007), 1003--1036.


\bibitem[CDG03]{cdg}
Campillo, A., Delgado, F. and Gusein-Zade, S.M.:
The Alexander polynomial of a plane curve singularity via the ring of functions on it,
{\em Duke Math J.} {\bf 117} (2003), no. 1, 125--156.


\bibitem[CDGZ04]{CDGPs} Campillo, A.,  Delgado, F. and Gusein-Zade, S. M.:
Poincar\'e series of a rational surface singularity, {\em Invent. Math.} {\bf 155} (2004),
no. 1, 41--53.

\bibitem[CDGZ08]{CDGEq}  Campillo, A.,  Delgado, F. and Gusein-Zade, S. M.:
Universal abelian covers of rational
surface singularities and multi-index filtrations,
{\em Funk. Anal. i Prilozhen.} {\bf 42} (2008), no. 2, 3--10.

\bibitem[CK14]{CK} Can, M.B. and Karakurt, \c{C}.: Calculating Heegaard--Floer homology by counting lattice points in tetrahedra, {\em Acta Math. Hungar.} {\bf 144 (1)} (2014), 43--75.


\bibitem[D78]{D78} Durfee, A.H.: The signature of smoothings of complex surface singularities,
{\em Math. Ann.} {\bf 232} (1978), no. 1, 85--98.

\bibitem[EN85]{EN} Eisenbud, D. and Neumann, W.: Three--dimensional link theory and invariants of plane curve singularities,
{\em Princeton Univ. Press} (1985).

\bibitem[FS90]{FS}Fintushel, R. and Stern, R.J.: Instanton homology of Seifert fibred homology three spheres, {\em Proc. London Math. Soc.} {\bf (3)} 61 (1990), no. 1,  109--137.



\bibitem[GrS93]{Gstrum} Gritzmann, P. and Sturmfels, B.: Minkowski addition of polytopes: computational complexity and applications to Gr\"obner bases, {\em SIAM J.Disc.Math} {\bf 6} (1993), no. 2, 246--269.

\bibitem[Kh78]{Khov} Khovanskii, A. G.: Newton polyhedra, and the genus of complete intersections (Russian), {\em Funktsional.
Anal. i Prilozhen} {\bf 12} (1) (1978), 51--61.

\bibitem[K76]{Kouch} Kouchnirenko, A.G.: Poly\'edres de Newton et nombres de Milnor, {\em Invent. Math.} {\bf 32} (1) (1976), 1--31.



\bibitem[LN14]{LN} L\'aszl\'o, T. and N\'emethi, A.: Ehrhart theory of polytopes and Seiberg-Witten invariants of plumbed 3--manifolds,
{\em Geometry and Topology} {\bf 18} (2014), no. 2, 717--778.



\bibitem[LNN17]{LNN} L\'aszl\'o, T., Nagy, J. and N\'emethi, A.: Surgery formulae for the Seiberg--Witten invariant of plumbed 3-manifold,
{\em arXiv:1702.06692 [math.GT]} (2017).

\bibitem[LSz16]{LSz} L\'aszl\'o, T. and Szil\'agyi,  Zs.:
 On Poincaré series associated with links of normal surface singularities,
{\em arXiv:1503.09012v2 [math.GT]} (2015).


\bibitem[LSz17]{LSzdiv} L\'aszl\'o, T. and Szil\'agyi,  Zs.:
 N\'emethi's division algorithm for zeta-functions of plumbed 3-manifolds,
{\em arXiv:1708.01093 [math.GT]} (2017).

\bibitem[Lim00]{Lim} Lim, Y.: Seiberg--Witten invariants for 3--manifolds in the case $b_1=0$ or $1$,
{\em Pacific J. of Math.} {\bf 195} (2000), no. 1, 179--204.



\bibitem[Les96]{Lescop} Lescop, C.: Global surgery formula for the Casson--Walker
invariant, {\em Ann. of Math. Studies}  {\bf 140}, Princeton Univ. Press, 1996.

\bibitem[MT80]{MT}
Merle, M. and  Teissier, B.:
Conditions d'adjonction, d'apr\`es DuVal,  in {\em S\'eminaire sur les Singularit\'es des Surfaces},
 Lecture Notes in Math., {\bf 777} (1980), Springer, Berlin, 229--245.

\bibitem[Mor84]{Mor} Morales, M.: Poly\`edre de Newton et genre g\'eom\'etrique d’une singularit\'e intersection compl\`ete, {\em Bull.
Soc. Math. France} {\bf 112} (3) (1984), 325--341.

\bibitem[N99]{Nfive} N\'emethi, A.: Five lectures on normal surface singularities,
lectures at the Summer School in {\em Low dimensional topology} Budapest,
Hungary, 1998; Bolyai Society Math. Studies {\bf 8} (1999), 269--351.

\bibitem[N05]{NOSZ} N\'emethi, A.: On the Ozsv\'ath--Szab\'o invariant of negative
definite plumbed 3--manifolds,
{\em Geometry and Topology} {\bf 9} (2005), 991--1042.

\bibitem[N07]{trieste} N\'emethi, A.: Graded roots and singularities,
{\em Singularities in geometry and topology},  World
Sci. Publ., Hackensack, NJ (2007), 394--463.

\bibitem[N08]{NPS} N\'emethi, A.: Poincar\'e series associated with surface singularities, in Singularities I, 271--297,
{\em Contemp. Math.} {\bf 474}, Amer. Math. Soc., Providence RI, 2008.


\bibitem[N11]{NJEMS} N\'emethi, A.: The Seiberg--Witten invariants of negative definite plumbed 3--manifolds,
{\em J. Eur. Math. Soc.} {\bf 13} (2011), 959--974.


\bibitem[N12]{NCL} N\'emethi, A.: The cohomology of line bundles of splice--quotient singularities,
{\em Advances in Math.} {\bf 229} 4 (2012), 2503--2524.

\bibitem[NN02]{NN1} N\'emethi, A. and Nicolaescu, L.I.:
 Seiberg--Witten invariants and surface singularities,
{\em Geometry and Topology} {\bf 6} (2002), 269--328.


\bibitem[NO09]{NOk} N\'emethi, A. and Okuma, T.:
On the Casson invariant conjecture of Neumann--Wahl, {\em Journal of Algebraic Geometry}
{\bf 18} (2009), 135--149.



\bibitem[NW90]{NWcas}
Neumann, W. and Wahl, J.: Casson invariant of links of singularities,
{\em Comment. Math. Helvetici} {\bf 65} (1990), 58--78.

\bibitem[NW05]{NWsq}
Neumann, W. and Wahl, J.: Complete intersection singularities of splice type as universal abelian covers,
{\em Geom. Topol.} {\bf 9} (2005), 699--755.


\bibitem[Nic01]{Nic01} Nicolaescu, L.: Lattice points inside rational simplices and the Casson invariant of Brieskorn spheres,
{\em Geometriae Dedicata} {\bf 88} (2001), 37--53.

\bibitem[Nic04]{Nic04} Nicolaescu, L.: Seiberg--Witten invariants of rational homology $3$--spheres,
{\em Comm. in Cont. Math.} {\bf 6} no. 6 (2004), 833--866.

\bibitem[Ok90]{Okzeta} Oka M.: Principal zeta-function of nondegenerate complete intersection singularity, {\em J. Fac. Sci.
Univ. Tokyo Sect. IA Math.} {\bf 37} (1990), no. 1, 11--32.

\bibitem[Ok97]{OkNdicis} Oka M.: Non-degenerate complete intersection singularity, {\em Actualit\'es math\'ematiques}, Hermann \'Editeurs (1997).


\bibitem[O08]{Ok} Okuma, T.: The geometric genus of splice--quotient singularities,
{\em Trans. Amer. Math. Soc.} {\bf 360} 12 (2008), 6643--6659.


%


\bibitem[Sav02]{Sav} Saveliev, N.: Invariants of Homology 3-Spheres, {\em Encyclopaedia of Mathematical Sciences}  {\bf 140}, Springer-Verlag Berlin Heidelberg, 2002.

\bibitem[S15]{Baldur} Sigur\dh sson, B: The geometric genus and Seiberg--Witten invariant of Newton nondegenerate surface singularities, PhD. thesis,
Central European University, Budapest, 2015.

\bibitem[SzV03]{SzV} Szenes, A. and Vergne, M.: Residue formulae for vector partitions and Euler--Maclaurin sums,
{\em Advances in Appl. Math.} {\bf 30} (2003), 295--342.


\bibitem[V76]{Var} Varchenko, A.: Zeta-function of monodromy and Newton's diagram, {\em Invent. Math.} {\bf 37} (1976), no. 3, 253--262.


\end{thebibliography}
\end{document}